\documentclass{amsart}
\usepackage[T1]{fontenc}
\usepackage[utf8]{inputenc}
\usepackage{tikz-cd,amsmath,microtype,hyperref,amsfonts,amssymb,amsthm}
\usepackage[scr=rsfs]{mathalfa}
\theoremstyle{plain}
\newtheorem{theorem}[equation]{Theorem}
\newtheorem*{claim1}{Claim 1}
\newtheorem*{claim2}{Claim 2}
\newtheorem{proposition}[equation]{Proposition}
\newtheorem{lemma}[equation]{Lemma}

\newtheorem{corollary}[equation]{Corollary}
\theoremstyle{definition}

\newtheorem{example}[equation]{Example}
\newtheorem{situation}[equation]{}
\newtheorem{remark}[equation]{Remark}
\renewcommand{\mathbb}{\mathbf}
\numberwithin{equation}{section}
\title{On Vlasenko's formal group laws}
\author{Dingxin Zhang}

\begin{document}
\maketitle
\begin{abstract}
  Given a Laurent polynomial over a ring flat over \(\mathbb{Z}\), Vlasenko defines a formal
  group law. We identify this formal group law with a coordinate system of a
  formal group functor, prove its integrality. When the ``Hasse--Witt matrix''
  of the Laurent polynomial is invertible, Vlasenko defines a matrix by taking a
  certain \(p\)-adic limit. We show that this matrix is the Frobenius of the
  Dieudonné module of this formal group modulo \(p\).
\end{abstract}

\section*{Introduction}

On a flat \(\mathbb{Z}\)-algebra \(R\), given a Laurent polynomial \(f\) with
coefficients in \(R\),
Vlasenko~\cite{vlasenko:higher-hasse-witt, vlasenko:formal-groups-congruences}
introduced a formal group law
\begin{equation*}
  F_{f}
  \quad
  \text{(\textit{a priori}, \(F_{f}\) is only defined over
    \(R\otimes\mathbb{Q}\))}
\end{equation*}
using the coefficients of the Laurent polynomials \(f^{n}\), \(n=1,2,\ldots\).
While being very explicit and in some sense canonical, the definition of
\(F_{f}\) constitutes some notation. Thus we invite the reader
to consult~\S\ref{situation:setting-formal-group} for the precise definition.
In the two papers cited above, Vlasenko studied \(F_{f}\) and deduced many
arithmetic consequences that can be stated without mentioning \(F_{f}\).

In this short note, we shall report our understanding to these formal group laws. Two
topics will be discussed. Each topic constitutes a section in the main text. Let
us give a brief overview of our results.

\subsection*{(a) Relation with geometry of hypersurfaces in toric varieties}

The first task is to report the integrality of \(F_{f}\) and its relation with
Artin--Mazur formal group functors~\cite{artin-mazur:formal-groups}. Using toric geometry, we prove the following
result.

\begin{theorem}%
  \label{theorem:main-2}
  Let \(R\) be a flat \(\mathbb{Z}\)-algebra.
  \begin{enumerate}
  \item \textup{(\(=\) Corollary~\ref{corollary:coord})} If \(R\) is noetherian,
    the formal group law
    \(F_{f}\) is a coordinate system of an Artin--Mazur type formal group
    functor~\eqref{situation:sheaf-cohomlogy} on \(R\).
  \item
    \textup{(\(=\) Theorem~\ref{theorem:integrality}, also a theorem of Vlasenko)}
    Without assuming \(R\) to be noetherian, the formal group law \(F_{f}\) is a power
    series with coefficients in \(R\).
  \end{enumerate}
\end{theorem}

The proof of Theorem~\ref{theorem:main-2}(1) is a replica of Stienstra's
article~\cite{stienstra:formal-group-from-algebraic-varieties} with some inputs
from toric geometry. Item (2) will be a consequence of Item (1).

Some comments on Theorem~\ref{theorem:main-2}(2) are in order.
\begin{itemize}
\item Beukers and Vlasenko~\cite{beukers-vlasenko:dwork-crystal-1} proved
  Theorem~\ref{theorem:main-2}(2) when \(R\) is \(p\)-adically complete and the
  Hasse--Witt matrix (see~\eqref{eq:vlasenko-matrix}) of \(f\) is
  invertible.
\item If \(p\) is a prime number, and if \(R\) admits a lift of the absolute
  Frobenius operator of \(R/p\), Vlasenko~\cite{vlasenko:higher-hasse-witt}
  proves that \(F_{f}\) has coefficients in \(R_{(p)}\). In particular, Vlasenko
  is able to prove Theorem~\ref{theorem:main-2}(2) for polynomial rings
  \(R=\mathbb{Z}[x_1,\ldots,x_{N}]\). In fact, this is already enough to prove
  Theorem~\ref{theorem:main-2}(2) because we can prove the integrality of a
  ``universal'' case. We shall give an alternative treatment, and deduce
  Theorem~\ref{theorem:main-2}(2) as a consequence of
  Theorem~\ref{theorem:main-2}(1). Due to our use of toric geometry, the
  integrality result obtained is not as general as Vlasenko's, see the footnote
  for~\eqref{situation:setting-formal-group}.
\end{itemize}

The relation with Artin--Mazur functors was noted by
Vlasenko~\cite{vlasenko:higher-hasse-witt} when the toric scheme defined by
the Newton polytope of \(f\) is a projective space.
However, our result shows that \(F_{f}\) should rather be thought as a
coordinate system of the formal group functor associated with the \emph{ideal sheaf}
of the hypersurface defined by \(f\), instead of the Artin--Mazur formal group
functor of the hypersurface itself. See
Remark~\ref{remark:flatness} below.

\begin{remark}[We do not need the flatness of the hypersurface defined by \(f\)]%
  \label{remark:flatness}
  The Laurent polynomial \(f\) defines a hypersurface \(X\) in a suitable
  toric scheme \(\mathbb{P}\) over \(R\). In~\cite{stienstra:formal-group-from-algebraic-varieties}, \(f\)
  corresponds to a hypersurface in \(\mathbb{P}^{n}\). Stienstra imposed a running
  flatness hypothesis on this hypersurface. In our result, we \emph{do not} need
  the flatness of the hypersurface defined by \(f\). The reason is that in our
  theorem we do not identify the formal group \(F_{f}\) with the Artin--Mazur
  formal group of the hypersurface \(X\), but with a variant of Artin--Mazur
  formal group attached to its ideal sheaf (which is an
  invertible sheaf, always flat over \(R\)). In fact, this point already
  occurred in Stienstra's article~[loc.~cit.]. See
  Lemma~\ref{lemma:formal-lie-cohomology} and
  \S{}\ref{situation:finish-the-proof}.

  When \(X\) \emph{is} flat over \(R\), and when \(R = W(k)\) is the ring of
  Witt vectors of a perfect field of characteristic \(p > 0\), the isogeny class
  of the Cartier--Dieudonné module of the reduction of \(F_{f}\) mod \(p\) is
  the slope \(<1\) part of the rigid cohomology group (with proper support)
  \(\mathrm{H}^{\dim\mathbb{P}_{k}}_{\text{rig,c}}(\mathbb{P}_{k} - X_k)^{<1}\).
  See Remark~\ref{remark:relation-with-rigid-cohomology}. Therefore the explicit
  formal group law allows us to extract the information of the slope \(<1\) part
  of the Newton polygon of the rigid cohomology.
\end{remark}

At the end of \S\ref{sec:integrality} we revisit a theorem of Honda concerning
formal group laws from hypergeometric equations. We explain why his formal group
law is only integral over \(\mathbb{Z}_{(p)}\) for large \(p\), by relating his
equations with Picard--Fuchs equations of ``underdiagram deformations''. The
formal integrals of some special power series solutions to the latter ordinary differential equations are
logarithms of Vlasenko's group laws.

\subsection*{(b) Higher Hasse--Witt matrices and Frobenius operators}

Vlasenko~\cite{vlasenko:higher-hasse-witt} considered the matrices
\begin{equation*}
  (\alpha_{s})_{u,v \in \Delta^{\circ} \cap \mathbb{Z}^{d}} =
  \text{the coefficient of } t^{p^{s}v-u}\text{ in }f^{p^{s}-1}
\end{equation*}
over a \(p\)-adically complete torsion free ring \(R\). These matrices were
called ``higher Hasse--Witt matrices'' by Vlasenko. If \(\alpha_{s}\) are all
invertible modulo \(p\), then Vlasenko proves the \(p\)-adic limit
\begin{equation*}
  \alpha=\lim_{s\to\infty}\alpha_{s+1}(\alpha_s^{\sigma})^{-1}
\end{equation*}
exists. Vlasenko asked~[loc.~cit.] whether \(\alpha\) is a Frobenius matrix
acting on some crystal (which she did not specify). This question is confirmed
by~\cite[Remark~5.4]{beukers-vlasenko:dwork-crystal-1} using what they call the
``Dwork crystal'' (with a very mild constraint on the coefficients of \(f\)).
Huang--Lian--Yau--Yu~\cite{hlyy-hasse} also studied this question,
and they are able to answer Vlasenko's question assuming \(\Delta\) is a smooth,
very ample
polytope and \(f\) defines a smooth hypersurface in the toric variety defined
by \(\Delta\).

In Section~\ref{sec:recover-frobenius}, we give an alternative answer to
Vlasenko's question (without constraints on \(\Delta\) or \(f\)).
Let \(\Gamma_{f}\) be the mod \(p\) reduction of the formal group \(\Phi_{f}\)
mentioned in Theorem~\ref{theorem:main-2}. If \(\alpha_{1}\) is invertible mod
\(p\), we shall show that the Dieudonné module of \(\Gamma_{f}\) is isoclinic of
slope \(0\)~\eqref{lemma:criterion-ordinary}, and the \(p\)-adic limit matrix
\(\alpha\) is the Frobenius matrix of the Dieudonné module of \(\Gamma_{f}\).

\begin{theorem}[= Theorem~\ref{theorem:dieudonne-module-frob}]%
  \label{theorem-0.3}
  Assume that \(R\) is \(p\)-adically complete flat \(\mathbb{Z}\)-algebra. Let
  \(\mathbb{D}^{\ast}(\Gamma_f)\) be the (covariant) Dieudonné crystal of \(\Gamma_{f}\) on
  \(R/p\). Assume that the matrix \(\alpha_{1}\) (see above) is invertible. Then
  \(\alpha\) is a matrix of the Frobenius operation of \(\mathbb{D}^{\ast}(\Gamma_f)_{R}\).
\end{theorem}

Here, we view \(\mathop{\mathrm{Spf}}R\) as an inductive system divided power thickening of
\(R/p\), i.e., an ind-object in the big crystalline site
\(\mathrm{CRIS}(\mathrm{Spec}(R/p)/\mathbb{Z}_{p})\), and
\(\mathbb{D}^{\ast}(\Gamma_f)_{R}\) is the Zariski sheaf on
\(\mathop{\mathrm{Spf}}R\) defined by the crystal
\(\mathbb{D}^{\ast}(\Gamma_f)\) (via taking limit).

When the hypersurface \(X\) defined by \(f\) in a toric scheme is flat over
\(R\) (without assuming \(\alpha_1\) invertible), as mentioned in
Remark~\ref{remark:flatness} above, the Dieudonné module of \(\Gamma_{f}\) gives
the slope \(<1\) part of the hypersurface defined by \(f\). See also
Remark~\ref{remark:relation-with-rigid-cohomology}. The significance of
Theorem~\ref{theorem-0.3} is that when \(\alpha_1\) is invertible, we have a
purely combinatorial way to extract the unit root part of (the rigid cohomology
of) the reduction of \(X \to \mathrm{Spec}(R)\), even when the general fibers
are singular.

\subsection*{Acknowledgments}
Professor Masha Vlasenko sent me a list of suggestions and corrections, and
clarified some of my misconceptions. I would like to thank her for her
invaluable help.

I am also grateful to Tsung-Ju Lee, for his suggestions on Example~\ref{example:use-gkz}; to
Shizhang Li, for pointing out how to use \(\delta\)-rings in the proof of
Theorem~\ref{theorem:dieudonne-module-frob}; to Qixiao Ma and Luochen Zhao, for discussions on
formal groups; to Chenglong Yu, for answering my questions about his paper; and
to Jie Zhou, for discussions on GKZ systems.

\section{Integrality of Vlasenko's formal group laws}
\label{sec:integrality}
\begin{situation}%
  \label{situation:setting-formal-group}\textbf{Notation and conventions.}
  In this section we fix the following notation and conventions.
  Let \(R\) be a commutative ring flat over \(\mathbb{Z}\).
  Let
  \begin{equation*}
    f(t) = \sum_{u\in \mathbb{Z}^{d}}a_{u}t^{u}\in R[t_1,\ldots,t_{d},(t_1\cdots t_{d})^{-1}]
  \end{equation*}
  be a Laurent polynomial with coefficients in \(R\). Let \(\Delta\) be the
  Newton polytope of \(f\). Recall that \(\Delta\) is the convex hull in
  \(\mathbb{R}^{d}\) of \(\{w \in \mathbb{Z}^{d}: a_{w} \neq 0\}\).

  We shall assume that the dimension of \(\Delta\) equals \(d\)%
  \footnote{Vlasenko pointed out to me that her integrality proof does not require
    \(\Delta\) to be full dimensional as we assumed here. Thus our integrality
    proof is not as general as hers. Assuming her integrality theorem, the
    results in \S\ref{sec:recover-frobenius} can go through for an arbitrary
    \(\Delta\), except Remark~\ref{remark:relation-with-rigid-cohomology}, which
    requires the relation with algebraic geometry.}, and that the interior
  \(\Delta^{\circ}\) of \(\Delta\) contains at least one lattice point.

  Following Vlasenko, we define, for \(v,w \in \Delta^{\circ}\cap\mathbb{Z}^{d}\),
  \begin{equation*}
    L_{v,w}(\tau) = \sum_{\nu=1}^{\infty}  \beta_{v,w,\nu}\frac{\tau^{\nu}}{\nu},
  \end{equation*}
  where \(\beta_{v,w,\nu} \in R\) equals the coefficient of \(t^{\nu w - v}\) in the
  expansion of \(f(t)^{\nu-1}\). We define
  \begin{equation*}
    L_{v}(\tau_{w}:w \in \Delta^{\circ}\cap \mathbb{Z}^{d}) = \sum_{w\in \Delta^{\circ}\cap \mathbb{Z}^{d}}L_{v,w}(\tau_{w}).
  \end{equation*}
  Then
  \begin{equation}
    \label{eq:formal-logarithm}
    L = (L_{v}: v\in \Delta^{\circ} \cap \mathbb{Z}^{d}) \in (R\otimes \mathbb{Q})[\![\tau_{w}: w\in \Delta^{\circ}\cap \mathbb{Z}^{d}]\!]^{N}
  \end{equation}
  where \(N=\#\Delta^{\circ}\cap\mathbb{Z}^{d}\). Finally, we define an
  \(N\)-dimensional formal group law on \(R\otimes_{\mathbb{Z}} \mathbb{Q}\) by
  \begin{equation}
    \label{eq:formal-group-law}
    F_{f}(x,y) = L^{-1}(L(x) + L(y)) \in (R\otimes \mathbb{Q})[\![x,y]\!].
  \end{equation}
  Here
  \begin{equation*}
    x = (x_{w}: w\in \Delta^{\circ} \cap \mathbb{Z}^{d}), \quad
    y = (y_{w}: w\in \Delta^{\circ} \cap \mathbb{Z}^{d}).
  \end{equation*}
\end{situation}

M.~Vlasenko~\cite[Theorem~2]{vlasenko:higher-hasse-witt} proves that if the
Frobenius operator of \(R/p\) lifts to \(R\), then \(F\in R_{(p)}[\![x,y]\!]\).
In fact, one can use the argument in Remark~\ref{remark:noetherian-is-ok} below
to prove that \(F \in R[\![x,y]\!]\) without assuming Frobenii can be lifted.
However, we give a different argument without using Vlasenko's theorem, and
prove the integrality using a different argument based on formal group functors.

\begin{theorem}%
  \label{theorem:integrality}
  Let notation and conventions be as
  in~\textup{\ref{situation:setting-formal-group}}.
  Then \(F_{f}(x,y) \in R [\![x,y]\!]\).
\end{theorem}

\begin{remark}%
  [We can assume \(R\) is noetherian]
  \label{remark:noetherian-is-ok}
  The coefficients of \(f\) generates a finitely generated subring \(R'\) of
  \(R\). The coefficient expansions used in the definition are all contained in
  \(R'\), and the coefficients of the series \(L_{v,w}(\tau)\) are then in the
  power series \((R'\otimes \mathbb{Q})[\![\tau]\!]\), and \(F(x,y)\) lies in
  \((R'\otimes \mathbb{Q})[\![x,y]\!]\). If we replace \(R\) by \(R'\), and
  we can prove \(F_{f} \in R'[\![x,y]\!]\), then we automatically get
  \(F_{f}\in R [\![x,y]\!]\). Thus, it suffices to prove the theorem for \(R'\).
  The virtue of \(R'\) is that it is a quotient of a polynomial algebra over
  \(\mathbb{Z}\) with finitely many variables, hence is a noetherian ring.
\end{remark}

As noted in~\cite{vlasenko:higher-hasse-witt}, when the Laurent polynomial is of
a special form, the formal group law \(F_{f}\) provides a coordinate system to the
Artin--Mazur formal group of a hypersurface in a projective space. Our proof of
Theorem~\ref{theorem:integrality} is based on this observation:
using the method of
J.~Stienstra~\cite{stienstra:formal-group-from-algebraic-varieties}, we show \(F_{f}\)
is a coordinate system of a formal group functor related to a hypersurface in a
toric scheme.

The proof goes as follows.
\begin{enumerate}
\item Construct a toric scheme over \(R\) and a relatively ample divisor
  \((\widetilde{f}=0)\) using \(\Delta\) and \(f\)~\eqref{situation:notation-toric}.
\item Prove the formal group functor defined by the \emph{ideal sheaf} of
  \((\widetilde{f}=0)\) is a formal Lie
  group~\eqref{lemma:formal-lie-cohomology}.
\item Using Čech cohomology, find an explicit logarithm of this formal Lie
  group over \(R \otimes \mathbb{Q}\)~\eqref{situation:formal-logarithm}.
\item Prove that this formal logarithm agrees with the one
  in~\ref{situation:setting-formal-group}~\eqref{situation:finish-the-proof}.
\end{enumerate}

While Theorem~\ref{theorem:integrality} assumes \(R\) to be a flat
\(\mathbb{Z}\)-algebra, some of the results needed in the proof (e.g., the
smoothness of a certain formal group functor) are valid over an arbitrary ring.
Thus in the sequel we will be careful about the hypotheses.

\begin{situation}%
  \label{nil-algebra}
  We begin by recalling the notion of formal group functors.
  Let \(R\) be a ring.
  A (necessarily non-unital) \(R\)-algebra \(A\) is a \emph{nil \(R\)-algebra} if for
  any \(a \in A\), there exists \(r \geq 0\), such that \(a^r = 0\). Let
  \(\mathfrak{Nil}_{R}\) be the category of nil \(R\)-algebras. Let \(A_i\)
  \(i\in I\) be nil \(R\)-algebras. The direct sum
  \(\bigoplus_{i\in I} A_i\) is a nil algebra with multiplication
  \((a_i: i\in I) \cdot (a'_i : i \in I) = (a_i a'_i: i\in I)\)
  (\(a_i, a'_i \in A_i\), and only finitely many \(a_i \neq 0\)).
\end{situation}

\begin{situation}%
  \label{situation:formal-group}
  A \emph{commutative formal group functor},
  or simply a \emph{formal group functor}, or simply a \emph{formal group}
  (\emph{in this note, all formal groups are assumed to be commutative}), on a
  ring \(R\) is a functor
  \(\Phi:\mathfrak{Nil}_{R}\to\mathrm{Mod}_{\mathbb{Z}}\). Usually one imposes
  some exactness conditions such as the functor is asked to preserve direct
  sums, or to be ``exact''. We shall not impose these conditions. In
  Lemma~\ref{lemma:left-exactness-artin-mazur}, we will establish an exactness
  property that we need later.

  The simplest formal group functor is the formal additive group
  \begin{equation*}
    \widehat{\mathbb{G}}_{\text{a}} : \mathfrak{Nil}_{R} \to \mathrm{Mod}_{\mathbb{Z}},
    \quad A \mapsto (A,+)
  \end{equation*}
  which simply forgets the multiplication on \(A\).

  We say a formal group \(\Phi\) is a \emph{formal Lie group} if it is naturally
  isomorphic, as \emph{set} valued functors, to some
  \(\widehat{\mathbb{G}}_{\text{a}}^{n}\). An isomorphism of set-valued functors
  \(\widehat{\mathbb{G}}_{\text{a}}^{n}\to\Phi\) is called a
  \emph{coordinate system} of \(\Phi\). The group structure on \(\Phi\) defines,
  by transport of structures, a group structure on the ideal
  \((x_1,\ldots,x_n)\) of the ring of power
  series \(R[\![x_1,\ldots,x_n]\!]\) as the ideal \((x_1,\ldots,x_n)\) is an
  inverse limit of nil \(R\)-algebras. This gives rise to a power series
  \(F(x,y)\) subject to the axioms of a formal group law. Therefore, a formal
  group law is simply equivalent to a choice of a coordinate system on a formal
  Lie group.

  One important example of a formal Lie group is the formal multiplicative
  group, notation \(\widehat{\mathbb{G}}_{\text{m}}\), whose group law is given
  by the polynomial \(L(x,y) = 1-(1-x)(1-y) = x+y-xy\), which is the
  coordinate expansion of the usual multiplication at \(1\). The functorial
  definition of the formal multiplicative group is to send a nil algebra \(A\)
  over \(R\) to the multiplicative group \((A,\star)\), where for \(a,b\in A\),
  \(a\star b = a +b - ab\).
\end{situation}

\begin{situation}%
  \label{situation:sheaf-cohomlogy}
  Let \(\mathcal{I}\) be a sheaf of (possibly non-unital) \(R\)-algebras on an
  \(R\)-scheme \(S\). Let \(F\) be a formal group on \(R\). Then for each nil
  algebra \(A\), we can define a new sheaf of abelian groups by sheafifying the
  following presheaf
  \begin{equation*}
    F(\mathcal{I} \otimes_{R} A) : U \mapsto F(\mathcal{I}(U)\otimes_{R}A).
  \end{equation*}
  Taking the \(i\)th cohomology of this sheaf yields a formal group. When
  \(F=\widehat{\mathbb{G}}_{\text{m}}\), the above formal group is denoted by
  \(\Phi^{i}(X,\mathcal{I})\), thus:
  \begin{equation*}
    \Phi^{i}(X,\mathcal{I}):
    A \mapsto \mathrm{H}^{i}(X, \widehat{\mathbb{G}}_{\text{m}}(\mathcal{I}\otimes_{R}A)).
  \end{equation*}
  This is a variant of the deformation functor considered by
  Artin--Mazur~\cite{artin-mazur:formal-groups}. We shall call
  \(\Phi^i(X,\mathcal{I})\) the \emph{Artin--Mazur formal group functor} associated
  with \(\mathcal{I}\).

  Each \(R\)-module \(M\) defines a nil algebra subject to the condition
  \(m_1 \cdot m_2 = 0\) for all \(m_1,m_2 \in M\). Note that for such a nil
  algebra we have
  \(\widehat{\mathbb{G}}_{\mathrm{a}}(M)=\widehat{\mathbb{G}}_{\mathrm{m}}(M)\).
  The restriction of a formal
  functor to the subcategory of \(R\)-modules defines a functor called the
  \emph{tangent spcae} to \(\Phi\). If \(M\) is an \(R\)-module viewed as a
  nil algebra, then
  \begin{equation*}
    \Phi^{i}(X,\mathcal{I})(M) = \mathrm{H}^{i}(X,\widehat{\mathbb{G}}_{\mathrm{m}}(\mathcal{I}\otimes_{R} M))=
    \mathrm{H}^{i}(X,\mathcal{I}\otimes_{R}M).
  \end{equation*}
\end{situation}

\begin{lemma}%
  \label{lemma:left-exactness-artin-mazur}
  Let \(R\) be a noetherian ring.
  Let \(X\) be a finite type scheme over \(R\).
  Let \(\mathcal{I}\) be a coherent ideal
  sheaf of \(\mathcal{O}_{X}\), flat over \(R\). Assume further that for any
  ideal \(J\) of \(R\), \(\mathrm{H}^{i-1}(X \otimes_{R} R/J,\mathcal{I})=0\).
  Then \(\Phi^{i}(X,\mathcal{I})\) is a left exact functor. That is,
  for any exact sequence \(0 \to N_1 \to N_2 \to N_3 \to 0\) of nil algebras
  over \(R\) (exact as \(R\)-modules),
  \begin{equation*}
    0\to \Phi^{i}(X,\mathcal{I})(N_1) \to \Phi^{i}(X,\mathcal{I})(N_2) \to
    \Phi^{i}(X,\mathcal{I})(N_3)
  \end{equation*}
  is exact.
\end{lemma}

\begin{proof}
  Since \(\mathcal{I}\) is flat over \(R\), the sequence
  \begin{equation*}
    0 \to \mathcal{I} \otimes_R N_1 \to
    \mathcal{I} \otimes_R N_2 \to
    \mathcal{I} \otimes_R N_3 \to 0
  \end{equation*}
  is exact. Applying \(\widehat{\mathbb{G}}_{\mathrm{m}}\), we get the exact
  sequence
  \begin{equation*}
    0 \to \widehat{\mathbb{G}}_{\mathrm{m}}(\mathcal{I} \otimes_R N_1) \to
    \widehat{\mathbb{G}}_{\mathrm{m}}(\mathcal{I} \otimes_R N_2) \to
    \widehat{\mathbb{G}}_{\mathrm{m}}(\mathcal{I} \otimes_R N_3) \to 0.
  \end{equation*}
  Applying cohomology groups, we get an exact sequence
  \begin{equation*}
    \Phi^{i-1}(X,\mathcal{I})(N_3) \to
    \Phi^{i}(X,\mathcal{I})(N_1) \to
    \Phi^{i}(X,\mathcal{I})(N_2) \to
    \Phi^{i}(X,\mathcal{I})(N_3).
  \end{equation*}
  Therefore, it suffices to prove \(\Phi^{i-1}(X,\mathcal{I})(N)=0\) for any nil
  algebra \(N\). Since \(N\) is a filtered colimit of finitely generated nil
  algebras, and since (on a noetherian topological space) taking the Zariski cohomology group of sheaves of abelian
  groups commutes with filtered colimits, it suffices to assume that \(N\) is a
  finitely generated nil algebra. Each such algebra \(N\) fits into a sequence
  \begin{equation*}
    N = N_0 \to N_1 \to \cdots \to N_r = 0
  \end{equation*}
  in which \(N_j \to N_{j+1}\) is surjective, and the kernel is generated by a
  single element \(\epsilon\), such that \(\epsilon^2 = 0\). Therefore, by using
  induction on \(r\), and using the exact sequences of cohomology, the vanishing
  of \(\Phi^{i-1}(X,\mathcal{I})(N)\) follows from the vanishing of
  \(\Phi^{i-1}(X,\mathcal{I})(R\epsilon)\), where \(\epsilon^2=0\). Let \(J\) be
  the annihilator of \(\epsilon\). Then (recall
  \eqref{situation:sheaf-cohomlogy} that
  \(\widehat{\mathbb{G}}_{\mathrm{m}}(M)=\widehat{\mathbb{G}}_{\mathrm{a}}(M)\)
  if \(M\) is an \(R\)-module viewed as a nil algebra)
  \begin{equation*}
    \Phi^{i-1}(X,\mathcal{I})(R\epsilon) =
    \mathrm{H}^{i-1}(X,\mathcal{I}\otimes_{R}R\epsilon) = \mathrm{H}^{i-1}(X\otimes_{R}R/J,\mathcal{I})=0.
  \end{equation*}
  This completes the proof.
\end{proof}

Next, we recall some toric geometry that we need. Our reference
is~\cite{cox-little-schneck:toric-varieties}. This reference treats only complex
toric varieties. But the parts related to fans, polytopes, and combinatorial
description of sheaf cohomology groups are also valid over \(\mathbb{Z}\) and
over any ring. The
part on vanishing theorems work for any algebraically closed field, and hence
the vanishing over a ring follows from an easy base change argument.

\begin{situation}%
  \label{situation:notation-toric}
  Let \(R\) be an arbitrary ring.
  Let \(f \in R[t_1,\ldots,t_d,(t_1\cdots t_{d})^{-1}]\)  be a Laurent
  polynomial.
  Let \(\Delta \subset \mathbb{Z}^{d}=:M\) be the Newton polytope of
  \(f\).
  Let \(\Sigma \subset N = M^{\vee}\) be the normal fan of \(\Delta\).
  Let \(\Sigma(1)\) be the set of 1-cones of
  \(\Sigma\). We set up the following notation
  (see~\cite[Chapter~5]{cox-little-schneck:toric-varieties} for more about the
  Cox ring and homogeneous coordinates).
  \begin{enumerate}
  \item \(z_{\rho} : \mathbb{A}^{\Sigma(1)} \to \mathbb{A}^{1}\) is the
    coordinate function with respect to the 1-cone \(\rho\).
  \item For a cone \(\sigma \in \Sigma\),
    \(\widehat{z}_{\sigma}=\prod_{\rho\notin\sigma}z_{\rho}\).
  \item \(Z(\Sigma)=\mathrm{Zeros}\{\widehat{z}_{\sigma} : \sigma \in \Sigma\}\subset\mathbb{A}^{\Sigma(1)}\).
  \item \(U(\Sigma) = \mathbb{A}^{\Sigma(1)} \setminus Z(\Sigma)\).
  \item \(\mathbb{D}(\Sigma)\) is the algebraic torus associated with
    \(\mathrm{Cl}(\Sigma)\) defined by the exact sequence
    \begin{equation*}
      0 \to M \to \mathbb{Z}^{\Sigma(1)} \to \mathrm{Cl}(\Sigma) \to 0.
    \end{equation*}
  \item \(\mathbb{P}_{\Sigma}\) is the toric variety associated with the
    fan \(\Sigma\), thus we have a \(\mathbb{D}(\Sigma)\)-torsor
    \(\pi: U(\Sigma) \to \mathbb{P}_{\Sigma}\).
  \item \(S = S(\Sigma) = R[z_{\rho} : \rho \in \Sigma(1)]\) is the
    ``Cox ring'' of \(\mathbb{P}_{\Sigma}\).
  \end{enumerate}
  Note that the Cox ring \(S(\Sigma)\) receives a
  \(\mathrm{Cl}(\Sigma)\)-grading by the \(\mathbb{D}(\Sigma)\)-action.
  The Laurent polynomial \(f\) then has a grading \(\beta\) with respect to the
  action of \(\mathrm{Cl}(\Sigma)\), and therefore defines a relative Cartier divisor
  \(X\) on \(\mathbb{P}_{\Sigma}\).
\end{situation}

The following result is well-known in toric geometry. One can obtain it by
applying Serre duality
to~\cite[Proposition~5.4.1]{cox-little-schneck:toric-varieties}. Since its proof
is needed in the proof of Lemma~\ref{lemma:formal-lie-cohomology} below, we feel obliged to sketch the proof.

\begin{lemma}%
  \label{proposition:compute-cartier-divisor}
  Notation be as in~\textup{\ref{situation:notation-toric}}.
  Let \(Y\) be an effective Cartier divisor of \(\mathbb{P}_{\Sigma}\) whose
  divisor class is \(\beta \in \mathrm{Cl}(\Sigma)\).
  Then there is an isomorphism
  \begin{equation*}
    \left( \frac{\widetilde{f}}{\prod_{\rho\in\Sigma(1)}z_{\rho}}R[z_{\rho}^{-1}:\rho\in\Sigma(1)] \right)_{-\beta}
    \cong \mathrm{H}^{d}(\mathbb{P}_{\Sigma},\mathcal{O}_{\mathbb{P}_{\Sigma}}(-Y)).
  \end{equation*}
  where \(\widetilde{f} \in S\) is the defining equation of \(Y\).
\end{lemma}

\begin{proof}
  Let
  \(U_{\sigma} = \{z \in \mathbb{A}^{\Sigma(1)}: \widehat{z}_{\sigma}\neq 0\}\),
  Then \(U_{\sigma}\) is stable under the action of \(\mathbb{D}^{\ast}_{\sigma}\),
  and \(U_{\sigma\cap \sigma'} = U_{\sigma} \cap U_{\sigma'}\).
  Let \(V_{\sigma}\) the image of \(U_{\sigma}\) in \(\mathbb{P}_{\Sigma}\).
  Then
  \begin{equation}
    \label{eq:covering}
    \mathfrak{V}=\{V_{\sigma}:\sigma\text{ maximal cone in }\Sigma\}
  \end{equation}
  is an open covering of \(\mathbb{P}_{\sigma}\).

  Let \(\mathcal{F}\) be a quasi-coherent sheaf on \(\mathbb{P}_{\Sigma}\). Then
  the cohomology of \(\mathcal{F}\) can be computed using the (alternating) Čech
  cohomology with respect to the covering \(\mathfrak{V}\). The Čech complex is
  \begin{equation*}
    \prod_{\sigma\text{ maximal}} \mathcal{F}(V_{\sigma}) \to
    \prod_{\sigma\text{ codim }1} \mathcal{F}(V_{\sigma}) \to
    \cdots \to
    \prod_{\rho \in \Sigma(1)} \mathcal{F}(V_{\rho}) \to
    \mathcal{F}(T)
  \end{equation*}
  where \(T\) is the embedded torus.

  In our case, \(\mathcal{F}=\mathcal{O}_{\mathbb{P}_{\Sigma}}(-\beta)\). Then
  \(\mathcal{F}(V_{\sigma})=(S[\widehat{z}_{\sigma}^{-1}])_{-\beta}\) (the
  homogeneous piece of degree \(-\beta\)). The proposition then follows from a
  straightforward computation.
\end{proof}

Now we construct an explicit coordinate system for
\(\Phi^{d}(\mathbb{P}_{\Sigma},\mathcal{O}_{\mathbb{P}_{\Sigma}}(-X))\).

\begin{lemma}%
  \label{lemma:formal-lie-cohomology}
  Notation be as in~\textup{\ref{situation:notation-toric}}. Assume that \(R\)
  is noetherian.
  The formal group
  \(\Phi^{d}(\mathbb{P}_{\Sigma},\mathcal{O}_{\mathbb{P}_{\Sigma}}(-X))\) is a
  formal Lie group.
\end{lemma}

\begin{proof}
  Following Stienstra, we prove this by constructing a coordinate.
  Let \(\widetilde{f}\) be the homogeneous equation that cuts out \(X\).

  We shall use the Čech cohomology to define an isomorphism (as functor of sets)
  between
  \(\Phi^{d}(\mathbb{P}_{\Sigma},\mathcal{O}_{\mathbb{P}_{\Sigma}}(-X))\)
  and the \(N\)-fold self-product of the formal
  additive group, where \(N\) is the \(R\)-rank of
  \begin{equation}
    \label{eq:ok}
    \left( \frac{\widetilde{f}}{\prod_{\rho\in\Sigma(1)}z_{\rho}}R[z_{\rho}^{-1}:\rho\in\Sigma(1)] \right)_{-\beta}.
  \end{equation}

  To organize the combinatorics, for each function
  \(\mathbf{m}: \Sigma(1) \to \mathbb{Z}_{\geq 1}\), write
  \begin{equation*}
    z^{\mathbf{m}} = \prod_{\rho\in\Sigma(1)}z_{\rho}^{\mathbf{m}(\rho)}.
  \end{equation*}
  Then an element in~\eqref{eq:ok} is written as
  \begin{equation*}
    \sum \lambda_{\mathbf{m}} \frac{F}{z^{\mathbf{m}}}.
  \end{equation*}
  Thus for each nil \(R\)-algebra \(A\), we identify \(A^{N}\) with the set of
  \(N\)-uples \((a_{\mathbf{m}})\). Then we define a map
  \begin{equation*}
    A^{N} \to \widehat{\mathbb{G}}_{\text{m}}(\mathcal{O}_{\mathbb{P}_{\Sigma}}(-X)\otimes_{R}A)(T)
  \end{equation*}
  sending \((a_{\mathbf{m}})\) to
  \begin{equation*}
    \sum_{\mathbf{m}} \frac{\widetilde{f}}{z^{\mathbf{m}}}\otimes a_{\mathbf{m}}
  \end{equation*}
  the summation being taken using the group structure of
  \[\widehat{\mathbb{G}}_{\text{m}}(\mathcal{O}_{\mathbb{P}_{\Sigma}}(-X)\otimes_{R}A)(T).\]
  We claim the composition
  \begin{equation*}
    A^{N} \to \widehat{\mathbb{G}}_{\text{m}}(\mathcal{O}_{\mathbb{P}_{\Sigma}}(-X)\otimes_{R}A)(T)
    \to \Phi^{d}(\mathbb{P}_{\Sigma},\mathcal{O}_{\mathbb{P}_{\Sigma}}(-X))(A)
  \end{equation*}
  is an isomorphism. Here, the second arrow is to take the cohomology class via
  the \v{C}ech complex used in the proof
  of~Lemma~\ref{proposition:compute-cartier-divisor}. The theorem then follows from
  this claim.

  Indeed, Lemma~\ref{proposition:compute-cartier-divisor} shows that this
  morphism (of set-valued functors on nil algebras) induces an isomorphism on
  the tangent space. Since the functor \(A\mapsto A^{N}\) is smooth and exact,
  the claim then follows
  from~\cite[Theorem~2.30]{zink:cartier-theory-formal-groups} in view of
  Lemma~\ref{lemma:left-exactness-artin-mazur}. The vanishing needed in
  Lemma~\ref{lemma:left-exactness-artin-mazur} is ensured by the
  Batyrev--Borisov vanishing
  theorem~\cite[Theorem~9.2.7]{cox-little-schneck:toric-varieties}. Since we
  have started with a polytope \(\Delta\), and the divisor \(X\) is linearly
  equivalent to the divisor \(D_{\Delta}\) described
  in~\cite[(4.2.7)]{cox-little-schneck:toric-varieties}, the amplitude needed in
  the vanishing is ensured
  by~\cite[Propnsition~6.1.10(a)]{cox-little-schneck:toric-varieties}. Although
  \cite[Theorem~9.2.7]{cox-little-schneck:toric-varieties} as stated requires
  the toric variety to be defined over the field of complex
  numbers, its proof is combinatorial and works for any field. The result over a
  base ring then follows from the theorem on cohomology and base change.
\end{proof}

\begin{situation}%
  \label{situation:formal-logarithm}
  \textbf{Construction of logarithm.} In this paragraph, we provide an explicit isomorphism between
  \(\Phi^{d}(\mathbb{P}_{\Sigma},\mathcal{O}_{\mathbb{P}_{\Sigma}}(-X))\) and a
  product of additive groups over \(R\otimes \mathbb{Q}\), following the method
  of Stienstra. For this purpose, we could replace \(R\) by the
  \(\mathbb{Q}\)-algebra \(R\otimes\mathbb{Q}\). Thus,
  \emph{in this paragraph, we will assume \(R\) is a noetherian \(\mathbb{Q}\)-algebra}.

  Recall that the usual logarithm defines, for each nil algebra \(A\) over
  \(R\), an isomorphism of abelian groups
  \begin{equation*}
    \ell(A) : \widehat{\mathbb{G}}_{\text{m}}(A) \to \widehat{\mathbb{G}}_{\text{a}}(A) ,
    \quad
    a \mapsto \sum_{n=1}^{\infty} \frac{1}{n}a^{n}
  \end{equation*}
  We shall use \(\ell_A\) to define an explicit isomorphism
  between
  \(\Phi^{d}(\mathbb{P}_{\Sigma},\mathcal{O}_{\mathbb{P}_{\Sigma}}(-X))\) and a
  product of formal additive group.

  We have the following commutative diagram
  \begin{equation*}
    \begin{tikzcd}
      & \widehat{\mathbb{G}}_{\text{m}}(\mathcal{O}_{\mathbb{P}_{\Sigma}}(-X)\otimes_{R}A)(T) \ar[r,"\ell(A)"]\ar[d]
      &  \widehat{\mathbb{G}}_{\text{a}}(\mathcal{O}_{\mathbb{P}_{\Sigma}}(-X)\otimes_{R}A)(T) \ar[d]\\
      A^{N} \ar[r,swap,"u(A)"] \ar[ru] & \Phi^{d}(\mathbb{P}_{\Sigma},\mathcal{O}_{\mathbb{P}_{\Sigma}}(-X))(A)
      \ar[r,swap,"\ell_X(A)"] \ar[r] & \widehat{\mathbb{G}}_{\text{a}}^{N}(A)
    \end{tikzcd}.
  \end{equation*}
  In the diagram, the right square is a diagram of abelian groups, whereas the
  left triangle is a diagram of sets. The vertical maps are ``taking the
  cohomology class'' of Čech cocycles. The map \(\ell_X(A)\) is induced by the
  logarithms on the Čech cochains with respect to the
  covering~\eqref{eq:covering} after taking cohomology. It is an isomorphism
  since \(\ell(A)\) induces chain level isomorphisms. The map \(u\) is the
  coordinate we constructed in the proof of
  Lemma~\ref{lemma:formal-lie-cohomology}. The composition \(\ell_X\circ u\) is
  then the ``coordinate representation'' of the formal logarithm of
  \(\Phi^{d}(\mathbb{P}_{\Sigma},\mathcal{O}_{\mathbb{P}_{\Sigma}}(-X))\).

  Let \(L_{X} = \ell_X \circ u\).
  By chasing the diagram, for each \((a_{\mathbf{m}})\) as in the proof of
  Lemma~\ref{lemma:formal-lie-cohomology}, we have
  \begin{equation*}
    L_{X}(a_{\mathbf{m}}) = \text{ the cohomology class of }\sum_{\mathbf{w}: \Sigma(1) \to \mathbb{Z}_{\geq 1}}
    \sum_{\nu=1}^{\infty} \frac{1}{\nu}\frac{\widetilde{f}^{\nu}}{z^{\nu\cdot \mathbf{w}}} \otimes a_{\mathbf{w}}^{\nu}.
  \end{equation*}
  To spell out the ``cohomology class'', we use
  Lemma~\ref{proposition:compute-cartier-divisor}. We should only look at those
  monomials in the expansions \(\widetilde{f}^{\nu-1}\) which has the following properties
  \begin{itemize}
  \item has degree \(\beta(\nu-1)\), thus \(\mathbf{w}\) must be of degree \(\beta\),
  \item part of the monomial ``cancels'' the denominator \(z^{\nu \mathbf{w}}\),
    and
  \item the rest part of the monomial produces a monomial \(z^{-\mathbf{v}}\)
    for some \(\mathbf{v}\) of degree \(\beta\).
  \end{itemize}
  Thus, to get the terms with contributions, we define, for an integer \(\nu\),
  and for maps \(\mathbf{v},\mathbf{w}:\Sigma(1)\to \mathbb{Z}_{\geq 1}\),
  \(\deg(\mathbf{v})=\deg(\mathbf{w})=\beta\).
  \begin{equation*}
    \beta_{\mathbf{v},\mathbf{w},\nu} = \text{coefficient of }z^{\nu\mathbf{w}-\mathbf{v}}
    \text{ in the expansion of } \widetilde{f}^{\nu-1}.
  \end{equation*}
  Then for each \(\mathbf{v}\) of degree \(\beta\), the
  contribution of \(\widetilde{f}/z^{\mathbf{v}}\) is given by
  \(\sum_{\nu,\mathbf{w}} \beta_{\mathbf{v},\mathbf{w},\nu}\otimes a_{\mathbf{w}}^{\nu}/\nu\).
  Therefore we can write the logarithm as
  \begin{equation*}
    L_{X}(a_{\mathbf{w}}) = (L_{\mathbf{v}}(a_{\mathbf{w}})),
  \end{equation*}
  where
  \begin{equation*}
    L_{\mathbf{v}}(a_{\mathbf{w}}:w :\Sigma(1)\to\mathbb{Z}_{\geq 1})
    = \sum_{\mathbf{w}} \sum_{\nu=1}^{\infty} \beta_{\mathbf{v},\mathbf{w},\nu} \frac{a_{\mathbf{w}}^{\nu}}{\nu}
  \end{equation*}
  The formal group law for the formal group we constructed is
  therefore
  \begin{equation*}
    F(X,Y) = L_{X}^{-1}(L_{X}(X_{\mathbf{v}}) + L_{X}(Y_{\mathbf{w}}))
  \end{equation*}
  which is a matrix of formal power series in
  \(X_{\mathbf{v}}, Y_{\mathbf{w}}\).
\end{situation}

\begin{situation}%
  \label{situation:finish-the-proof}
  In this paragraph we finish the proof of Theorem~\ref{theorem:integrality}.
  By Remark~\ref{remark:noetherian-is-ok}, we can assume \(R\) is a noetherian
  ring flat over \(\mathbb{Z}\). In this case, we shall
  prove the formal logarithms defined in~\ref{situation:formal-logarithm}
  agrees with the series~\eqref{eq:formal-logarithm}.

  Recall that we begin with the Newton polytope \(\Delta\) of the Laurent
  polynomial \(f\). The normal fan \(\Sigma\) of \(\Delta\) defines
  \(\mathbb{P}_{\Sigma}\) and the lattice gives rise to a torus invariant
  relative Cartier divisor
  \(D=\sum_{\rho\in\Sigma(1)}a_{\rho}D_{\rho}\)~\cite[Equation~(4.2.7)]{cox-little-schneck:toric-varieties}. The
  lattice \(\Delta\) can be recovered from the numbers \(a_{\rho}\)
  \begin{equation*}
    \Delta = \{w \in \mathbb{R}^{d} : \langle w ,u_{\rho} \rangle \geq -a_{\rho}, \forall \rho \in \Sigma(1)\},
  \end{equation*}
  by \cite[Proposition~6.1.10, Theorem~6.2.1, Exercise~4.3.1]{cox-little-schneck:toric-varieties},
  and the Laurent polynomial can be viewed as a section of the Cartier divisor
  \(\mathcal{O}_{\mathbb{P}_{\Sigma}}(D)\). Moreover,
  \begin{equation*}
    \mathrm{H}^{0}(\mathbb{P}_{\Sigma},\mathcal{O}_{\Sigma}(D))\cong R^{\Delta\cap \mathbb{Z}^{d}},
  \end{equation*}
  by \cite[Example~4.3.7]{cox-little-schneck:toric-varieties}. Thus the
  information about \(\Delta\) is equivalent to the information about \(\Sigma\)
  and the relative Cartier divisor \(D\).

  Next, write the Laurent polynomial \(f\) into
  \(\sum_{u\in\Delta\cap\mathbb{Z}^{d}}x_{u}t^{u}\). Then the Laurent monomial
  \(t^{u}\) corresponds to the monomial
  \begin{equation*}
    z^{\langle u,D \rangle} = \prod_{\rho\in\Sigma(1)} z_{\rho}^{\langle u,u_{\rho} \rangle+ a_{\rho}}
  \end{equation*}
  where \(u_{\rho}\) is the smallest generator of the ray \(\rho \in \Sigma(1)\)
  in the dual space of \(\mathbb{R}^{d}\). Thus \(f\) gives rise to an element
  \begin{equation*}
    \widetilde{f}(z) = \sum x_{u} z^{\langle u,D \rangle}
  \end{equation*}
  in the Cox ring \(S\).
  It is homogeneous of degree \(\beta=[D]\) in \(S\),
  and defines an effective Cartier divisor \(X\) on
  \(\mathbb{P}_{\Sigma}\).
  Moreover, the correspondence \(f \leftrightarrow \widetilde{f}\) establishes a bijection
  \begin{equation*}
    \mathrm{H}^{0}(\mathbb{P}_{\Sigma},\mathcal{O}_{\sigma}(D))\cong S_{\beta},
  \end{equation*}
  see~\cite[Proposition~5.4.1(b)]{cox-little-schneck:toric-varieties}. Under this
  correspondence, the expansion coefficients in \(\widetilde{f}^{m}\) are the same as the
  expansion coefficients of \(f^{m}\) since for the monomial
  \(\mathbf{w}:\Sigma(1)\to\mathbb{Z}_{\geq 1}\), homogeneous of degree
  \(\beta\), \(z^{\mathbf{w}}\) corresponds
  to \(t^{w}\) for some \(w \in \Delta \cap \mathbb{Z}^{d}\). This finishes the
  proof of Theorem~\ref{theorem:integrality}.\qed
\end{situation}

\begin{corollary}[of the proof of Theorem~\ref{theorem:integrality}]%
  \label{corollary:coord}
  Assume that \(R\) is noetherian. The formal group law \(F_{f}\) considered by
  Vlasenko~\eqref{eq:formal-group-law} is a coordinate system of the formal
  group functor
  \(\Phi^{d}(\mathbb{P}_{\Sigma},\widehat{\mathbb{G}}_{\mathrm{m}}(\mathcal{O}_{\mathbb{P}_{\Sigma}}(-X)))\).
\end{corollary}

\textit{From now on, we shall use \(\Phi_{f}\) to denote the formal Lie group over a
ring \(R\) determined by the formal group law \(F_{f}\).}

\medskip
The integrality of \(\Phi_{f}\) could be used to explain the integrality of some
formal group laws considered by Professor
T.~Honda~\cite{honda:formal-groups-from-hypergeometric-functions}.
Let \(N\) be an integer. Honda considered the
generalized hypergeometric ordinary differential equation
\begin{equation}
  \label{eq:honda}
  \left(\tau^{N}\prod_{\theta\in S} (\delta + N\theta) - \delta^{|S|}\right)g(\tau) = 0.
\end{equation}
where \(\delta=\tau\partial_{\tau}\), \(S\) is a subset of
\(\{1/N,\ldots,(N-1)/N\}\). Let \(g(\tau)=\sum_{n\geq 0} A(n)\tau^{Nn}\) be the
generalized hypergeometric function which is the only solution
to~\eqref{eq:honda} at \(0\). Set \(f(x) = \int_{0}^{x}g(\tau)\mathrm{d}\tau\).

\begin{theorem}[Honda~\cite{honda:formal-groups-from-hypergeometric-functions}]%
  \label{situation:honda-theorem}
  Suppose that \(\{N\theta : \theta \in S\}\) contains all the reduced residues
  mod \(N\), then \(F(x,y)=f^{-1}(f(x)+f(y))\), a priori a rational power
  series, actually lies in \(\mathbb{Z}_{(p)}[x,y]\), for every \(p > N\).
\end{theorem}

Let us indicate how Honda's integrality is related to the integrality of
\(F_{f}\).

\begin{situation}%
  \label{situation:subdiagram-deformation}
  In the sequel, we assume \(R\) is flat over \(\mathbb{Z}\), and assume we have
  fixed an embedding of \(R\otimes \mathbb{Q}\) into \(\mathbb{C}\).
  Let \(f\) be a Laurent polynomial with Newton polytope \(\Delta\).

  For each \(w \in \Delta^{\circ} \cap \mathbb{Z}^{d}\), consider the
  1-parameter family of hypersurfaces in the \(\mathbb{P}_{\Sigma}\)
  \begin{equation*}
    X_{w}(\tau) = \overline{\text{Zeros}(t^{w} + \tau f(t))} \subset \mathbb{P}_{\Sigma},
  \end{equation*}
  called an ``underdiagram deformation''. This deformation then determines a
  ``Picard--Fuchs system'' on \(\mathbb{A}^{1}_{\mathbb{C}}\). By definition, a
  differential operator is said to be a Picard--Fuchs operator if it annihilates
  the cohomology class of a differential form on the generic \(X_{w}\). The
  Picard--Fuchs system is the cyclic \(\mathscr{D}\)-module obtained by dividing
  the left ideal generated by Picard--Fuchs operators.
\end{situation}

The relation between solutions of differential equations and the formal group
\(\Phi_{f}\) is based on the following trivial observation.

\begin{lemma}%
  \label{lemma:formal-derivative-solution-picard-fuchs}
  In the situation above, the formal power series
  \begin{equation*}
    \sum_{\nu=1}^{\infty} \beta_{v,w,\nu}\tau^{\nu-1}
  \end{equation*}
  (the derivative of a series appears in the formal logarithms of \(F_{f}\)
  considered in~\textup{\ref{situation:formal-logarithm}})
  is a formal power series solution to the Picard--Fuchs system around
  \(\tau=0\).
\end{lemma}

\begin{proof}
  For each \(v \in \Delta^{\circ}\cap \mathbb{Z}^{d}\), we have a standard volume
  form on the torus given by
  \begin{equation*}
    \Theta_{v} = t^{v} \frac{\mathrm{d}t_1}{t_1}\wedge \cdots \wedge\frac{\mathrm{d}t_d}{t_d}
  \end{equation*}
  Then we know the cohomology class (in the logarithmic cohomology of the
  complement of \(t^w + \tau f = 0\))
  \begin{equation*}
    \frac{\Theta_v}{t^w + \tau f(t)}
  \end{equation*}
  (or the residue of it) satisfies the Picard--Fuchs equation by fiat. We then
  integrate the form along the standard top homology cycle of \(T\) to get an
  expansion with respect to \(\tau\). This is equivalent to taking the degree
  \(0\) term of the fraction
  \begin{align*}
    t^{v-w}(1+\tau t^{-w}f(t))^{-1}
    &= t^{v-w}\sum_{\nu=1}^{\infty}\tau^{\nu-1}t^{-(\nu-1)w}f(t)^{\nu-1} \\
    &= t^{v-w}\sum_{\nu=1}^{\infty} \tau^{\nu-1} t^{-(\nu-1)w}\sum_{w_1,\ldots,w_{\nu-1}} x_{w_1}\cdots x_{w_{\nu-1}} t^{w_1+\cdots + w_{\nu-1}}.
  \end{align*}
  Hence the constant term is
  \begin{equation*}
    \sum_{\nu} \sum_{\substack{w_1,\ldots,w_{\nu-1}\\\sum w_i = \nu v-w}} x_{w_1}\cdots x_{w_{\nu-1}} \tau^{\nu-1} =\sum_{\nu=1}^{\infty} \beta_{v,w,\nu} \tau^{\nu-1}.
  \end{equation*}
  as desired.
\end{proof}

It is not very easy to produce the precise ordinary differential equations
directly from the Laurent polynomial \(f\) we start with. Relatively easier is
to produce ordinary differential equations whose solutions \emph{contain} the
``period integrals'' using the so-called GKZ
systems~\cite{gelfand-kapranov-zelevinski:hypergeometric-functions-and-toric-varieties}.

In the following example, the GKZ system is simple enough so that we can
easily get the Picard--Fuchs equations out of them.

\begin{example}%
  \label{example:use-gkz}
  We consider the polytope \(\Delta\) generated by (minimal) lattice points
  \(u_1,\ldots,u_{n}\) in \(\mathbb{Z}^{n}\) subject to the only relation
  \begin{equation*}
    \sum_{i=1}^{n} q_{i}u_{i} = 0.
  \end{equation*}
  Note that \(\Delta\) is not the polytope of a weighted projective space.
  It is its \emph{face fan} that defines the weighted projective space
  \(\mathbb{P}(q_1,\ldots,q_n)\). Let \(N=\sum q_i\). We assume that
  \(q_{i}\mid{N}\) so that \(\Delta\) is a reflexive polytope. We shall
  consider the GKZ system associated with the matrix
  \begin{equation*}
    A =
    \begin{bmatrix}
      1 & 1 & 1 & \cdots &1 \\
      | & | & | & \cdots &| \\
      0 & u_1 & u_1 &\cdots & u_n \\
      | & | & | & \cdots &|
    \end{bmatrix}.
  \end{equation*}
  The GKZ system is a \(\mathscr{D}\)-module on the ``space of coefficients''
  which is the affine space \(\mathbb{A}^{n+1}\) with coordinate system
  \((a_0,a_1,\ldots,a_n)\). Each point \((a_0,\ldots,a_n)\) corresponds to a
  Laurent polynomial \(a_0 + \sum_{i=1}^{n} a_i t^{u_i}\). To get the desired
  Picard--Fuchs equation, we shall descend the ``modified box operator''
  \begin{equation*}
    \square a_{0}^{-1} = \partial_0^{N} a_{0}^{-1} - \prod_{i=1}^{n}\partial_{i}^{q_i}a_{0}^{-1}
  \end{equation*}
  via \(\tau^{N} = \prod_{i=1}^{n} a^{q_i}_{i}/a^{N}_0\),
  to the underdiagram deformation
  \(1 + \tau \sum_{i=1}^{n} t^{u_i}\)
  (\(a_{0}^{-1}\tau^{N}\) is killed by all the ``Euler operators'' in the GKZ system
  with parameter \((-1,0,\ldots,0)\)). By computation we get
  \begin{equation*}
    \mathcal{L}_{\text{GKZ}} = (-\tau)^{N} \prod_{i=1}^{N}(\delta+i) - \prod_{i=1}^{n}\prod_{j=0}^{q_{i}-1}\left(\frac{q_i}{N}\delta-j\right)
    \quad (\delta = \tau\partial_{\tau}).
  \end{equation*}
  Using the commuting relation
  \(\tau^{N} \delta = \delta \tau^{N} - N\tau^{N}\), we have a factorization
  \(\mathcal{L}_{\mathrm{GKZ}} = P(\delta) \mathcal{L}_{\mathrm{PF}}\).

  We thus get an ordinary differential equation \(\mathcal{L}_{\text{PF}}g(\tau)=0\).
  Sometimes, after scaling the variable \(\tau \to \pm N\tau\),
  \(\mathcal{L}_{\mathrm{PF}}\) changes to a
  differential operator considered by Honda. In these situations,
  by Lemma~\ref{lemma:formal-derivative-solution-picard-fuchs} above, the unique
  power series solution of \(\mathcal{L}_{\text{PF}}g(\tau) = 0\) should give
  rise to an \emph{integral formal group law} (Theorem~\ref{theorem:integrality}), and
  will imply Honda's integrality, while Honda's
  formal group law is only integral in \(\mathbb{Z}_{(p)}\) for \(p \nmid N\)
  due to the scaling of \(\tau\).

  We give two special cases illustrating our point.
  The first is to consider the Laurent polynomial
  \(f(t) = \sum_{i=1}^{d-1}t_{i} + \frac{1}{t_1\cdots t_{d-1}}\), which corresponds
  to the case \((q_1,\ldots,q_{d})=(1,1\ldots,1)\).
  Then the only underdiagram deformation is given by \(1 + \tau f(t)\).
  From the GKZ system we can infer an ordinary differential equation
  \begin{equation*}
    \mathcal{L}_{\text{GKZ}}=(-d)^{d}\tau^{d} \prod_{i=1}^{d} (\delta + i) - \delta^{d}
    = \delta \cdot ((-d)^{d}(\delta-1) \cdots (\delta-d+1)\tau^{d} - \delta^{d-1}).
  \end{equation*}
  This differential operator has a factor
  \[\mathcal{L}_{\text{PF}}=(-d)^{d}\tau^{d}(\delta+1)\cdots(\delta+d-1)-\delta^{d-1}.\]
  This is almost of Honda's type. The difference is a constant factor
  The related equation considered by Honda
  is obtained from \(\mathcal{L}_{\text{PF}}g(\tau)=0\) by making a substitution
  \(\tau\leftrightarrow (-d)\tau\). Therefore the related Honda's formal group
  law is only \(\mathbb{Z}_{(p)}\)-integral for those \(p \nmid d\).

  As another example, consider \(N=4\), and \((q_1,q_2,q_3)=(1,1,2)\). The above
  method produces the operator
  \begin{equation*}
    \mathcal{L}_{\mathrm{PF}}=4^4\tau^{4}(\delta+1)(\delta+3) - \delta^2.
  \end{equation*}
  which is related to Honda's under the correspondence
  \(\tau \leftrightarrow 4\tau\).
\end{example}

\section{Recovering unit-root Frobenius when Hasse--Witt is invertible}
\label{sec:recover-frobenius}
Let \(p\) be a prime number. Let \(R\), \(f\), \(\Delta\), and \(\Sigma\) be as
in~\ref{situation:notation-toric}. We assume in addition that \(R\) admits a
lift \(\sigma\) of the absolute Frobenius of \(R/p\). Let us consider
\begin{equation}
  \label{eq:vlasenko-matrix}
  (\alpha_{s})_{u,v \in \Delta^{\circ} \cap \mathbb{Z}^{d}} =
  \text{the coefficient of } t^{p^{s}v-u}\text{ in }f^{p^{s}-1}.
\end{equation}
In~\cite{vlasenko:higher-hasse-witt}, these matrices were called ``higher
Hasse--Witt matrices'' of \(f\). Vlasenko proved the following result
concerning these matrices.

\begin{proposition}[Vlasenko~{[loc.~cit.]}]
  \label{proposition:vlasenko-congruences}
  Let notation be as above. Then we have
  \begin{enumerate}
  \item For every \(s \geq 1\),
    \(\alpha_s\equiv\alpha_1\cdot\alpha_1^{\sigma}\cdots\alpha_1^{\sigma^{s-1}}\mod p\).
  \item Assume \(\alpha_{1}\) (hence any of the \(\alpha_{s}\)) is invertible
    over \(R\), and \(R\) is \(p\)-adically complete. Then
    \(\alpha_{s+1}(\alpha_{s}^{\sigma})^{-1}\equiv\alpha_{s}(\alpha_{s-1}^{\sigma})^{-1}\mod
    p^{s}\).
  \end{enumerate}
\end{proposition}

In the situation of Proposition~\ref{proposition:vlasenko-congruences}(2), the
\(p\)-adic limit
\begin{equation}
  \label{eq:frobenius-as-limit}
  \alpha=\lim_{s\to\infty}\alpha_{s+1}(\alpha_s^{\sigma})^{-1}
\end{equation}
exists. Vlasenko conjectured in [loc.~cit.] that \(\alpha\) is a matrix of some
Frobenius operation on some \(F\)-crystal. Under a very mild assumption on the
coefficients of \(f\),
Beukers--Vlasenko~\cite[Remark~2.5]{beukers-vlasenko:dwork-crystal-1} identifies
\(\alpha\) with a matrix of the Frobenius operation of their
``Dwork crystal''. Assuming the hypersurface \(X\) we mentioned in
Section~\ref{sec:recover-frobenius} is smooth over the base, the ambient
toric variety is smooth, and a technical condition on \(\Delta\),
the paper of Huang--Lian--Yau--Yu~\cite{hlyy-hasse} identifies \(\alpha\) with the Frobenius of the
unit root part of the relative crystalline cohomology of the family.

The purpose of this section is to prove that, under the
assumption of Proposition~\ref{proposition:vlasenko-congruences}(2), \(\alpha\)
is a matrix of the Frobenius operation of the (covariant) Dieudonné crystal of
the formal group \(\Gamma_{f}\), the reduction of \(\Phi_{f}\) modulo \(p\).
At the end of this section (Remark~\ref{remark:relation-with-rigid-cohomology})
we explain how to relate the Dieudonné crystal with
the geometric isocrystal associated with the hypersurfaces in the toric variety
defined by \(\Delta\).


\begin{situation}[An overview of Cartier's theory of curves]%
  \label{situation:cartier-theory-notation}
  We need the some basic facts in Cartier's curve theory. A thorough reference
  is Lazard's book~\cite{lazard:commutative-formal-groups}. We shall follow
  Lazard's conventions and notation, and explain some part of the theory we find
  necessary to understand the rest of the note.
  \begin{enumerate}
  \item A \emph{curve} on a formal Lie group \(G\) is simply a morphism (of
    set-valued functors) \(\gamma\) from \(\widehat{\mathbb{G}}_{\mathrm{a}}\)
    into \(G\).
    If \(R\) is a ring of characteristic \(0\), so that \(R\) embeds in to
    \(R \otimes \mathbb{Q}\), then we can represent a curve using the formal
    logarithm
    \(\log_{G}:G_{R\otimes\mathbb{Q}}\to{\mathbb{G}}_{\mathrm{a}}^{r}\), thus
    identify a curve with a power series
    \(\log_{G}\gamma(t)=\sum_{i=1}^{\infty}a_{i}t^{i}\) where
    \(a_i\in R^r\otimes\mathbb{Q}\).
    \label{item:curve}
  \item The tangent space of \(G\) could be regarded as a free \(R\)-module.
    A \emph{basic set of curves} of \(G\) is a collection of curves
    \(\gamma_1,\ldots,\gamma_{r}\) on \(G\) such that their tangent vectors
    form a basis of the tangent space of \(G\).
    If \(\star\) is the group operation of \(G\), the morphism (of set
    valued functors)
    \begin{equation*}
      \widehat{\mathbb{G}}_{\mathrm{a}}^{r} \xrightarrow{\sigma} G,
      \quad (t_1,\ldots,t_r) \mapsto \gamma_1(t_1) \star \cdots \star \gamma_{r}(t_r)
    \end{equation*}
    is an isomorphism, and it provides a coordinate system of \(G\). The group
    law under this coordinate system is called a
    \emph{curvilinear formal group law}.
    \label{item:basic-set-curve-curvilinear}

    Suppose that \(R\) is a \(\mathbb{Q}\)-algebra. Then in terms of the
    notation above, we have
    \begin{equation*}
      (\log_G \circ \sigma)(t_1,\ldots,t_r) = \sum_{i=1}^{r} \log_{G}\gamma_{i}(t_i)
    \end{equation*}
    is literally the sum of formal power series.
  \item Vlasenko's formal group law \(F_{f}\) for \(\Phi_{f}\) is a example of a
    curvilinear formal group law, defined by the curves \(\ell_{w}\),
    \(w \in \Delta^{\circ}\cap \mathbb{Z}^{d}\), where
    \begin{equation*}
      (\log_{\Phi_f} \circ \ell_{w})(\tau) = \sum_{\nu=1}^{\infty} \beta_{w,\nu} \frac{\tau^{\nu}}{\nu}
    \end{equation*}
    in which \(\beta_{w,\nu}= (\beta_{v,w,\nu})\),
    see~\ref{situation:setting-formal-group}.
    \label{item:vlasenko-curve}
  \item Among all curves on a formal Lie group \(G\) on a
    \(\mathbb{Z}_{(p)}\)-algebra there is a special class that is most relevant
    to our discussion. These are the \emph{\(p\)-typical curves} on \(G\). We
    shall not give the precise definition of \(p\)-typical curves. For our
    discussion, it is useful to know that if \(R\) has characteristic \(0\),
    then a curve \(\gamma\) on \(G\) is \(p\)-typical if and only if the power
    series \(\log_{G}\circ \gamma\) is of the form \(\sum_{s} a_{s}t^{p^s}\),
    i.e., in the power series expansion, only \(t\) to some \(p\)-power has
    possibly nonzero coefficients. The abelian group of all \(p\)-typical curves on
    \(G\) is denoted by \(\mathcal{C}(G)\).
    \label{item:p-typical-curve}
  \item For each ring \(R\), Cartier defined a noncommutative ring \(E(R)\). The
    ring \(E(R)\)
    consists of ``operators'' on the \(p\)-typical curves. Therefore, for each
    formal Lie group \(G\) over \(R\), \(\mathcal{C}(G)\) is a left
    \(E(R)\)-module. (see~\cite[IV~\S2]{lazard:commutative-formal-groups}).
    When \(R\) is a
    perfect field of characteristic \(p>0\),
    the \(E(R)\)-module \(\mathcal{C}(G)\) consisting of \(p\)-typical curves on
    \(G\) is
    also called the (covariant) \emph{Cartier--Dieudonné module} (or simply the
    Dieudonné module) of \(G\).
    \label{item:cartier-ring}
  \item The curves \(\ell_v\) described in Item~\eqref{item:vlasenko-curve} are
    not \(p\)-typical. The \(p\)-typical component \(\gamma_v\) of \(\ell_v\) is
    determined by
    \begin{equation*}
      (\log_{\Phi_f}\circ \gamma_v)(\tau) = \sum_{s=0}^{\infty} \alpha_{v,s} \frac{\tau^{p^s}}{p^s}
    \end{equation*}
    (see~\eqref{eq:vlasenko-matrix} and Item~\eqref{item:p-typical-curve} above). Clearly they form a basic set of
    \(p\)-typical curves on \(\Phi_{f}\). For formal Lie groups over a
    \(\mathbb{Z}_{(p)}\)-algebra, using \(p\)-typical curves is sufficient to
    determine the formal group
    (see~\cite[Chapter~IV]{lazard:commutative-formal-groups}).
    \label{item:p-typical-curve-vlasenko}
  \end{enumerate}
\end{situation}

\begin{example}(Homothety, Verschebung, and Frobenius)%
  \label{example:additive-operators}
  There are three operators in \(E(R)\).
  For each \(a \in R\), we can define an element \([a] \in E(R)\):
  \(([a]\gamma)(t)=\gamma(at)\), called the homothety operator of \(a\).  There
  is also the \(p\)-typical Verschebung operator \(V\) and the \(p\)-typical
  Frobenius operator \(F\). For our purposes, we only need to know the effect of
  the operators \(V\) and \([a]\) on additive curves, summarized below.

  When \(G\) is the additive group
  \(\widehat{\mathbb{G}}_{\mathrm{a}}^{r}\), a \(p\)-typical curve on \(G\) is
  given by a power series
  \begin{equation*}
    \gamma(t) = \sum_{i=0}^{\infty} a_{i}t^{p^i} \quad (a_i \in R^r),
  \end{equation*}
  and the abelian group structure of
  \(\mathcal{C}(\widehat{\mathbb{G}}_{\mathrm{a}}^r)\) is simply the addition of
  power series. We have
  \begin{align*}
    (F\gamma)(t) &= \sum_{i=0}^{\infty} pa_{i+1} t^{p^i} \\
    (V\gamma)(t) &= \sum_{i=1}^{\infty} a_{i-1} t^{p^i}, \\
    ([c]\gamma)(t) &= \sum_{i=0}^{\infty} c^{p^i} a_{i} t^{p^i}.
  \end{align*}
\end{example}

\begin{lemma}%
  \label{lemma:basics-on-curves}
  Let \(G\) be an \(r\)-dimensional formal Lie group over a
  \(\mathbb{Z}_{(p)}\)-algebra \(R\).
  \begin{enumerate}
  \item Let \(\gamma_1, \ldots,\gamma_r\) be a basic set of \(p\)-typical curves
    on \(G\). Then every \(p\)-typical curve \(\gamma\) can be
    uniquely written as
    \begin{equation*}
      \gamma = \sum_{n=0}^{\infty} V^{n} [x_{n,i}] \gamma_i.
    \end{equation*}
  \item Let \(\varphi:R \to R'\) be a ring homomorphism. Let \(\varphi_{\ast}:E(R) \to E(R')\)
    be the base-change homomorphism of the Catier rings. Let
    \(x=\sum_{i,j}V^{i}[x_{ij}]F^{j}\) be an element in \(E(R)\). Then
    \(\varphi_{\ast}(x)=\sum_{i,j}V^{i}[\varphi(x_{ij})]F^{j}\).
  \end{enumerate}
\end{lemma}

\begin{proof}
Item (1) is~\cite[IV~5.15, IV~5.17]{lazard:commutative-formal-groups}. Item (2)
is~\cite[IV~2.5]{lazard:commutative-formal-groups}.
\end{proof}

\begin{situation}%
  \label{situation:congruence-curve-coefficients}
  Lemma~\ref{lemma:two-basic-sets-of-curves} below will give a tool to produce congruence relations from
  the theory of curves. To state it, we make some hypotheses and set up
  some notation.
  \begin{enumerate}
  \item Let \(R\) be a flat \(\mathbb{Z}\)-algebra. Let \(\varphi: R \to R/p\)
    be the reduction mod \(p\) map. Let \(G\) be an \(r\)-dimensional formal Lie
    group over \(R\).
  \item Let \(\gamma_1,\ldots,\gamma_r\) and
    \(\gamma_1^{\ast},\ldots,\gamma_r^{\ast}\) be two basic sets of
    \(p\)-typical curves on \(G\) such that
    the tangent vectors satisfy \(\dot{\gamma}_i=\dot{\gamma}_{i}^{\ast}\).
    Assume further that
    \(\varphi_{\ast}\gamma_{i}=\varphi_{\ast}\gamma_{j}^{\ast}\).
  \item Let \(\log_{G}\) be the formal logarithm of \(G\) (after changing the
    base ring to \(R\otimes \mathbb{Q}\)). Then we can write
    \begin{equation*}
      (\log_G\circ \gamma_{j})(t) = \sum_{s=0}^{\infty} a_{j,s} \frac{t^{p^s}}{p^{s}}, a_{j} \in R^{r},
    \end{equation*}
    and similarly
    \begin{equation*}
      (\log_G\circ \gamma^{\ast}_{j})(t) = \sum_{s=0}^{\infty} a^{\ast}_{j,s} \frac{t^{p^s}}{p^{s}}, a^{\ast}_{j} \in R^{r},
    \end{equation*}
    see~\cite[V~8.19]{lazard:commutative-formal-groups}.
  \item  Finally let \(a_{s}\) to be
    the matrix \((a_{1,s},\ldots,a_{r,s})\) (similarly define \(a_{s}^{\ast}\)).

  \end{enumerate}
\end{situation}

\begin{lemma}%
  \label{lemma:two-basic-sets-of-curves}
  Let notation and conventions be as
  in~\textup{\ref{situation:congruence-curve-coefficients}}. We have
  \begin{equation*}
    a_{s}  \equiv a_{s}^{\ast} \mod p^{s}.
  \end{equation*}
\end{lemma}

\begin{proof}
  By Lemma~\ref{lemma:basics-on-curves}(1), we can write
  \begin{equation}
    \label{eq:curve-relation}
    \gamma_{j} = \sum V^{n}[x_{n,i}^{(j)}] \gamma_i^{\ast}.
  \end{equation}
  Since \(\varphi_{\ast}\gamma_j = \varphi_{\ast}\gamma_j^{\ast}\), we have by
  Lemma~\ref{lemma:basics-on-curves}(2) and Hypothesis
  \ref{situation:congruence-curve-coefficients}(2) that
  \begin{equation*}
    \varphi_{\ast}\gamma_{j}^{\ast} =
    \varphi_{\ast}\gamma_{j} = \sum V^{n}[\varphi(x_{n,i}^{(j)})] \varphi_{\ast}\gamma_i^{\ast}.
  \end{equation*}
  Applying the uniqueness part of Lemma~\ref{lemma:basics-on-curves}(1) (for the
  ring \(E(R/p)\)), we conclude that \(\varphi(x_{n,i}^{(j)}) = 0\) in \(R/p\)
  unless \(i=j\) and \(n=0\). In other words, \(p \mid x_{n,i}^{(j)}\) for all
  \((n,i,j)\neq(0,i,i)\). In view of
  Hypothesis~\ref{situation:congruence-curve-coefficients}(2),
  \(x_{0,i}^{(i)}=1\).

  Thanks to Lemma~\ref{lemma:two-basic-sets-of-curves}(2), if
  \(\psi : R \hookrightarrow R\otimes \mathbb{Q}\) is the inclusion of \(R\) in
  \(R\otimes \mathbb{Q}\), then~\eqref{eq:curve-relation} remains valid (by abuse
  of notation, we identify \(\gamma_{j}\) with \(\psi_{\ast}\gamma_j\)).
  Now we apply \(\log_{G}\) to the equality~\eqref{eq:curve-relation}. By
  the basic rules of \(V\) and \([a]\) described
  in Example~\ref{example:additive-operators}, we get
  \begin{align*}
    \sum_{s=0}^{\infty} a_{j,s} \frac{t^{p^s}}{p^s}
    &= \sum_{n=0}^{\infty} V^{n}[x_{n,i}^{(j)}] \left\{\sum_{s=0}^{\infty} a_{i,s}^{\ast} \frac{t^{p^s}}{p^s}\right\} \\
    &= \sum_{n=0}^{\infty} V^{n} \left\{ \sum_{s=0}^{\infty} (x_{n,i}^{(j)})^{p^s} a_{i,s}^{\ast} \frac{t^{p^s}}{p^s} \right\} \\
    &= \sum_{n=0}^{\infty} \sum_{s=n}^{\infty} (x_{n,i}^{(j)})^{p^{s-n}} a_{i,s-n}^{\ast} \frac{t^{p^s}}{p^{s-n}} \\
    &= \sum_{s=0}^{\infty} \sum_{n=0}^{s} (x_{n,i}^{(j)})^{p^{ s-n }} a_{i,s-n}^{\ast} \frac{t^{p^s}}{p^{s-n}}.
  \end{align*}
  It follows that in the ring \(R\otimes \mathbb{Q}\), we have
  \begin{equation*}
    a_{j,s} = \sum_{n=0}^{s} p^{n} (x_{n,i}^{(j)})^{p^{s-n}} a_{i,s-n}^{\ast}
  \end{equation*}
  As both sides fall in \(R\), the displayed equality is valid in \(R\) as well.
  Since \(p \mid x_{n,i}^{(j)}\) unless \(n=0\), \(i=j\), we see
  (remember that \(p^{s-n} > s-n\))
  \(p^{s} \mid p^{n}(x_{n,i}^{(j)})^{p^{s-n}} a_{i,s-n}^{\ast}\). This implies that
  \(a_{j,s} \equiv a_{j,s}^{\ast} \mod p^{s}\).
\end{proof}

We will also need the following simple fact.

\begin{lemma}%
  \label{lemma:nak}
  Let \(R\) be a ring. Let \(a\) be an element in \(R\).
  Assume that \(R\) is \(a\)-torsion free and \(a\)-adically complete.
  Let \(\varphi:R^{n}\to R^{n}\) be a morphism of free \(R\)-modules.
  Then \(\varphi\) is invertible if and only if the reduction of \(\varphi\)
  modulo \(a\) is invertible.
\end{lemma}

\begin{proof}
  Since \(\varphi\) is given by a linear map, it is continuous.
  The ``only if'' part is obvious. Let us prove the ``if'' part.
  Let \(x \in R^{n}\) be an element. Write \(x = x_0\).  Inductively
  define \(x_i\) and \(y_i\) so that \(x_{i}-\varphi(y_{i})\) is zero modulo
  \(a\), write \(x_i-\varphi(y_i)=ax_{i+1}\), and we define \(y_{i+1}\) such
  that \(\varphi(y_{i+1})\equiv x_{i+1}\) mod \(a\). Then the image of
  \begin{equation*}
    y = \sum_{i=0}^{\infty} a^{i}y_{i},
  \end{equation*}
  satisfies
  \begin{equation*}
    \varphi(y)=\sum_{i=0}^{\infty}a^{i}\varphi(y_i)=\sum_{i=0}^{\infty}(a^{i}x_{i}-a^{i+1}x_{i+1})=x_0=x.
  \end{equation*}
  This proves \(\varphi\) is surjective.

  To show \(\varphi\) is injective, suppose \(\varphi(x)=0\). Since the
  reduction of \(\varphi\) is injective, we must have \(a \mid x\). Then
  \(x=ax_1\) for some \(x_1\). Since \(R\) has no \(a\)-torsion, it follows that
  \(\varphi(x_1)=0\) as well. Continuing this way we see
  \(x \in \bigcap_{\nu=1}^{\infty}a^{\nu}R^{n}\). Since \(R\) is \(a\)-adically
  complete, in particular it is \(a\)-adically separated. Therefore \(x=0\) as
  desired.
\end{proof}

The last basic lemma we need concerning curve theory is a criterion for a formal
group to have an ``ordinary'' reduction modulo \(p\). Recall that a formal Lie
group over a perfect field of characteristic \(p\) is called ordinary, or of
codimension \(0\), or isoclinic of slope \(0\), if the semilinear Frobenius map
on the Dieudonné module is an isomorphism. Isoclinic modules of slope \(0\) are
the simplest in the spectrum of the Dieudonné--Manin classifications of crystals
over a perfect field. See~\cite[VI~\S6, \S8]{lazard:commutative-formal-groups}.

\begin{lemma}%
  \label{lemma:criterion-ordinary}
  Let \(R\) be complete discrete valuation ring of characteristic \(0\). Assume
  that its residue field \(k\) is perfect of characteristic \(p\). Let \(G\) be a formal group on
  \(R\) of dimension \(r\). Let \(\gamma_1, \ldots,\gamma_r\) be a basic set of
  \(p\)-typical curves on \(G\). Let \(\Gamma\) be the reduction of \(G\) modulo
  the maximal ideal of \(R\). Write
  \begin{equation*}
    \log_{G} \gamma_i = \sum_{i=0}^{\infty} a_{i,s} \frac{t^{p^s}}{p^s}, \quad
    a_{i,s} \in \text{ the tangent space of \(G\)}.
  \end{equation*}
  Assume that the matrix \(a_{1} = (a_{1,1},\cdots,a_{r,s})\) is a basis of
  the tangent space modulo the maximal ideal of \(R\). Then \(\Gamma\) is isolinic of slope
  \(0\).
\end{lemma}

\begin{proof}
  Fix a basis of the tangent space of \(G\), we can regard \(a_{i,s}\) as
  column vectors and \(a_{1}\) as a matrix. The condition is then \(a_{1}\) is
  an invertible matrix modulo the maximal ideal of \(R\).

  Let \(F\) be the \(p\)-typical Frobenius operator. We need to show that
  \(\mathcal{C}(\Gamma)=F\mathcal{C}(\Gamma)\)
  by~\cite[VI~7.5]{lazard:commutative-formal-groups}. Since \(a_{1}\) is
  invertible modulo the maximal ideal, it is invertible as a matrix with entries in \(R\)
  See Lemma~\ref{lemma:nak}.
  Since
  \begin{equation*}
    (\log_{G}\circ F\gamma_i)(t) = \sum_{i=0}^{\infty} a_{i+1,s} \frac{t^{p^s}}{p^{s}},
  \end{equation*}
  and since \(a_1\) is invertible, we see the tangent vectors of
  \(F\gamma_1,\ldots,F\gamma_{r}\) still generate the tangent space of \(G\).
  Therefore \(F\gamma_1,\ldots,F\gamma_r\) is also a basic set of curves. This
  means that \(\varphi_{\ast}F\gamma_1,\ldots,\varphi_{\ast}F\gamma_r\) is a
  basic curve of \(\Gamma\). To finish the proof, we recall that in \(E(k)\) the
  following relations hold:
  \begin{equation*}
    FV = VF, \quad F[a] = [a^{p}]F \quad (\forall a \in k).
  \end{equation*}
  Since \(F\gamma_1,\ldots,F\gamma_r\) is a basic set of curves, we can write
  any curve in \(\mathcal{C}(\Gamma)\) as
  \begin{align*}
    \gamma &= \sum_{n,i} V^{n}[x_{n,i}] F\gamma_{i} & (x_{n,i} \in k)\\
           &= \sum_{n,i} FV^{n}[x_{n,i}^{1/p}] \gamma_i & \in F\mathcal{C}(\Gamma)
  \end{align*}
  by Lemma~\ref{lemma:basics-on-curves}(1)
  (the second equality holds thanks to the perfectness of \(k\)).  This implies
  the desired equality \(F\mathcal{C}(\Gamma)=\mathcal{C}(\Gamma)\).
\end{proof}

\begin{example}%
  Set \(R = \mathbb{Z}_{p}\), whose Frobenius operation is the identity. Let us
  consider the Laurent polynomial \(f(t_1,t_2)=t_1+t_2+(t_{1}t_{2})^{-2}\). Its
  Newton polytope consists of two interior points \(u=(0,0)^{T}\) and
  \(v=(-1,-1)^{T}\). The following table summarizes the higher Hasse--Witt
  matrix \(\alpha_{1}\) of \(f\) with respect to \(u,v\).

  \medskip
  \begin{center}
    \begin{tabular}{|c|c|}
      \hline
      \(p^s\) & \(\alpha_s\)\\
      \hline
      \hline
      \(p^s=5k+1\) & \(\begin{bmatrix}\frac{(5k)!}{(k!)^{5}} & 0 \\ 0 & \frac{(5k)!}{(k!)^{2}(3k)!}\end{bmatrix}\)\\
      \(p^s=5k+2\) & \(\begin{bmatrix}0 & \frac{(5k+1)!}{(k!)^{2}(3k+1)!} \\ 0 & 0\end{bmatrix}\)\\
      \(p^s=5k+3\) & \(\begin{bmatrix}0 & 0 \\ \frac{(5k+2)!}{(k!)((2k+1)!)^2} & 0\end{bmatrix}\)\\
      \(p^s=5k+4\) & \(\begin{bmatrix}0 & 0 \\ 0 & 0\end{bmatrix}\)\\
      \hline
    \end{tabular}
  \end{center}

  \medskip
  Thus, the mod \(p\) reduction of \(\Phi_f\) is isoclinic of slope \(0\) if and
  only if \(p \equiv 1\mod 5\), and it is a direct sum of two multiplicative
  groups.
\end{example}

Now we can turn back to the group \(\Phi_{f}\). In the sequel, we denote the
reduction of \(\Phi_{f}\) modulo \(p\) by \(\Gamma_{f}\).

\begin{lemma}%
  \label{lemma:absolute-conjecture-proof}
  Let \(R=W(k)\) be the ring of Witt vectors of a perfect field \(k\) of
  characteristic \(p\). Let \(\Gamma_{f}\) be the reduction of \(\Phi_{f}\)
  modulo \(p\). Assume that the matrix \(\alpha_{1}\)~\eqref{eq:vlasenko-matrix}
  is invertible. Then the \(p\)-adic limit
  \begin{equation*}
    \lim_{s\to\infty} \alpha_{s+1}\cdot (\alpha_s^{\sigma})^{-1}
  \end{equation*}
  (see~\eqref{eq:vlasenko-matrix}) exists~\cite{vlasenko:higher-hasse-witt}, and
  is a matrix of the Frobenius operator \(\eta\) on the Cartier--Dieudonné
  module \(M=\mathcal{C}(\Gamma_f)\).
\end{lemma}

\begin{proof}
  We consider the generalized Lubin--Tate group \(\mathrm{LT}(M,\eta)\) à la
  Cartier~\cite{cartier:generalized-lubin-tate}. In general, for any formal
  group \(\Gamma_{f}\) of finite height over \(k\), the generalized Lubin--Tate
  group \(\mathrm{LT}(M,\eta)\) (\(M\) being the Cartier--Dieudonné module of
  \(p\)-typical curves on \(\Gamma_{f}\)) is the universal extended lift of
  \(\Gamma_{f}\) over \(W(k)\), see~\cite[VII~7.17]{lazard:commutative-formal-groups}.
  In our case, \(\Gamma_{f}\) is isoclinic of slope 0 (by Lemma~\ref{lemma:criterion-ordinary}),
  \(\mathrm{LT}(M,\eta)\) is therefore without additive kernels. Thus, for any
  lift \(\Phi\) of \(\Gamma_{f}\) over \(W(k)\), we always have an isomorphism
  \(\mathrm{LT}(M,\eta)\cong \Phi\). In particular, there is an isomorphism
  \(\mathrm{LT}(M,\eta)\cong\Phi_f\).

  The idea of the proof is as follows. The generalized Lubin--Tate group admits
  an explicit set of basic \(p\)-typical curves defined by the Frobenius
  operation, and our formal group \(\Phi_f\) admits an explicit set of basic
  \(p\)-typical curves defined by \(\alpha_s\). The above mentioned isomorphism
  will then provide a rule transforming the Frobenius related curve set to the
  expansion-coefficients related curve set. The limit formula will then be a
  consequence of Lemma~\ref{lemma:two-basic-sets-of-curves}.

  Let us carry out the above scheme. Define,
  for \(v \in \Delta^{\circ} \cap \mathbb{Z}^{d}\) a curve
  \(\gamma_{v} \in \mathcal{C}(\Phi_{f})\) by
  \begin{equation*}
    \log_{\Phi_f}\gamma_{v}(x) = \sum_{s=0}^{\infty} \alpha_{v,s} \frac{x^{p^s}}{p^{s}},
  \end{equation*}
  where
  \(\alpha_{v,s}=((\alpha_{s})_{u,v}:u\in\Delta^{\circ}\cap\mathbb{Z}^{d})\)~\eqref{eq:vlasenko-matrix}.
  In view of~\ref{situation:formal-logarithm}, each \(\gamma_{v}\) is a \(p\)-typical
  curve of the formal group functor \(\Phi_{f}\). Since \(\alpha_0\) is the
  identity matrix, the \(\gamma_{v}\)'s  form a basic set of
  \(p\)-typical curves. Since
  \(M=\mathcal{C}(\Gamma_{f})=E(k)\otimes_{E(W)}\mathcal{C}(\Phi_{f})\)~\cite[VII~6.8]{lazard:commutative-formal-groups},
  the images \(e_{v}\) of the curves \(\gamma_{v}\) form a basic set of curves
  of \(\Gamma_{f}\). In particular,
  \(\{e_{v}:v\in\Delta^{\circ}\cap\mathbb{Z}^{d}\}\) is a basis of the
  Cartier--Dieudonné module \(M\). For each \(e_{v}\), let \(\gamma_{v}^{\ast}\)
  be the curve in \(\mathcal{C}(\mathrm{LT}(M,\eta))\) given by
  \begin{equation*}
    \log_{\Phi_{f}} \gamma_{v}^{\ast}(x) = \sum_{s=0}^{\infty} \eta^{s}(e_{v}) \frac{x^{p^s}}{p^{s}} = \sum_{s=0}^{\infty}b_{v,s}\frac{x^{p^s}}{p^s}.
  \end{equation*}
  Then \(\gamma_{v}^{\ast}\) form a basic set of curves in
  \(\Phi_{f} \cong \mathrm{LT}(M,\eta)\),
  see~\cite[(8), (17) above, (19c)]{cartier:generalized-lubin-tate}.
  By construction,
  \(\{\gamma_{v}^{\ast}\}\) and \(\{\gamma_{v}\}\) are two basic sets of curves
  on \(\Phi_{f}\) which restrict to the same set of curves \(\{e_v\}\) of \(\Phi_{f_0}\).
  This enables us to apply Lemma~\ref{lemma:two-basic-sets-of-curves}, and we
  get a rather strong congruence relation:
  \begin{equation*}
    \alpha_{s} \equiv b_{s} \mod p^{s}.
  \end{equation*}
  Write \(\alpha_s = b_s + p^s c_s\).
  By definition and by that \(\eta\) is semilinear, \(b_{1}\) is the matrix
  of \(\eta\) with respect to the basis \(\{e_{v}\}\), and the matrix
  \(b_{s}\) is \(b_1 b_1^{\sigma} \cdots b_{1}^{\sigma^{s-1}}\).
  Note in particular we have \(b_{s+1}(b_{s}^{\sigma})^{-1} = b_1\).
  It follows that
  \begin{align*}
    \alpha_{s+1} (\alpha_{s}^{\sigma})^{-1}
    &= (b_{s+1} + p^{s+1}c_{s+1})(b_{s}^{\sigma})^{-1}(\mathrm{Id}+p^{s}c_{s}(b_{s}^{\sigma})^{-1})^{-1} \\
    &\equiv b_{s+1}(b_{s}^{\sigma})^{-1} \mod p^{s} \\
    &\equiv b_{1} \mod p^s.
  \end{align*}
  Therefore, \(p^{s} \mid (b_{1} - \alpha_{s+1}(\alpha^{\sigma}_s)^{-1})\).
  Thus the limit \(\alpha = \lim \alpha_{s+1}(\alpha^{\sigma}_s)^{-1}\) exists,
  and equals \(b_{1}\), which is the matrix of the Frobenius operation on
  \(M\) with respect to the basis \(\{e_{v}\}\).
\end{proof}

We have explained that, when the base is \(W(k)\), the limit matrix \(\alpha\)
is related to the Frobenius action on the Cartier--Dieudonné module of
\(\Gamma_{f}\). Now if \(R/p\) is not a perfect field, the analogue of the
Cartier--Dieudonné module of the formal group \(\Gamma_{f}\) is its
(covariant) Dieudonné crystal \(\mathbb{D}^{\ast}(\Gamma_{f})\) (which is the
contravariant Dieudonné crystal of the Cartier dual of \(\Gamma_{f}\)). The basic
reference for Dieudonné crystal
is~\cite{berthelot-breen-messing:crystalline-dieudonne-theory-2}. We shall not
review the theory of Dieudonné crystals. It suffices to know that the value
of \(\mathbb{D}^{\ast}(\Gamma_{f})\) on a ``perfect point''
\(x:\mathrm{Spec}(k) \to \mathrm{Spec}(R/p)\) (\(k\) is a perfect field) of
\(R/p\) is given by the Cartier--Dieudonné module of the fiber of
\(\Gamma_{f}\) over \(x\).

The following theorem shows that we can
identify the limit matrix \(\alpha\) with the Frobenius action on the value of
the Dieudonné crystal \(\mathbb{D}^{\ast}(\Gamma_{f})\) on \(R\). It turns out
we can reduce the general case to the special case treated before, by some
standard yoga.

\begin{theorem}%
  \label{theorem:dieudonne-module-frob}
  Let notation be as in~\textup{\ref{situation:notation-toric}}.
  Assume further that
  \begin{enumerate}
  \item \(R\) is \(p\)-adically complete, \(p\)-torsion free ring,
  \item \(R\) has a lifting \(\sigma\) of the absolute Frobenius of \(R/p\).
  \end{enumerate}
  Let \(\mathbb{D}^{\ast}(\Gamma_{f})\) be the (covariant) Dieudonné crystal of the reduction
  \(\Gamma_{f}\) of \(\Phi_{f}\) mod \(p\). Assume \(\alpha_{1}\) is invertible
  in \(R/p\). Then the \(p\)-adic limit \(\alpha\) is the Frobenius of the the
  \(R\)-module \(\mathbb{D}^{\ast}(\Gamma_f)_{R}\).
\end{theorem}

\begin{proof}
  By Lemma~\ref{lemma:nak}, \(\alpha_1\) itself is invertible in
  \(R\).
  As a first step, we assume both \(R\) and \(R/p\) are integral domains.
  In this step, we repeat a construction used by
  N.~Katz~\cite{katz:internal-reconstruction-unit-root-f-crystal-via-expansion}.
  Let \(A\) be the perfection of \(R/p\). Then there is a unique lifting of the
  inclusion \(R/p \to A\) to an inclusion \(R \to W(A)\) which sits in a
  commutative diagram
  \begin{equation*}
    \begin{tikzcd}
      R \ar[d,"\sigma"] \ar[hook,r] & W(A) \ar[d,"\sigma"] \\
      R \ar[r] & W(A)
    \end{tikzcd}
  \end{equation*}
  where the right vertical arrow, still denote by \(\sigma\), is the canonical
  Frobenius of \(W(A)\). Let \(K_0\) be the field of fractions of \(R/p\).
  In~\cite[\S7]{katz:internal-reconstruction-unit-root-f-crystal-via-expansion},
  it is shown that we have an injection \(R \to W(K_0^{\text{perf}})\). We have
  two matrices \(\alpha\) and \(\beta_1\), both have entries in \(R\). In order
  to show \(\alpha = \beta_{1}\), it suffices to prove it in a larger ring
  \(W(K_{0}^{\text{perf}})\). The result for this ring has been established in
  Lemma~\ref{lemma:absolute-conjecture-proof} above.

  We proceed to prove the theorem in its full generality. We shall use
  some simple properties about \(\delta\)-rings~(see for
  example~\cite[\S2]{bhatt-scholze:prism}). The upshot is that on a
  \(p\)-torsion free ring, having a \(\delta\)-ring structure is equivalent to fixing a
  lifting of the mod \(p\) Frobenius, and a \(\delta\)-ring homomorphism between
  \(p\)-torsion free rings is equivalent to a map preserving Frobenii.

  Set
  \(\mathcal{R}_0=\mathbb{Z}_{(p)}[x_{w}:w\in\Delta\cap \mathbb{Z}^{d}]\).
  The free
  \(\delta\)-ring with variables \(x_{w}\) is denoted by
  \(\mathcal{R}_1=\mathbb{Z}_{(p)}\{x_{w}:w\in\Delta\cap\mathbb{Z}^{d}\}\)
  (see~\cite[2.11]{bhatt-scholze:prism}). Abstractly, this is a polynomial ring
  with infinitely many variables.
  We use \(\phi\) to denote the Frobenius of \(\mathcal{R}_{1}\).

  Let \(\mathcal{R}\) be the \(p\)-adic completion of the localization
  \(\mathcal{R}_{1}[\phi^{m}(\det A_1)^{-1},m\geq 0]\). where \(A_{1}\) is the
  Hasse--Witt matrix for
  \(\mathbf{f}(t)=\sum_{w\in\Delta\cap\mathbb{Z}^{d}}x_{w}t^{w}\).
  Since we are localizing a system stable under the Frobenius, \(\mathcal{R}\)
  is a \(\delta\)-ring.

  \begin{claim1}
    The \(\delta\)-ring \(\mathcal{R}\) has the following universal property: suppose that we are
    given a ring homomorphism \(\varphi:\mathcal{R}_0 \to R\), where
    \begin{itemize}
    \item \(R\) is a \(p\)-adically complete, \(p\)-torsion free \(\delta\)-ring,
    \item \(\varphi(\det A_1)\) is invertible on \(R\),
    \end{itemize}
    there is a unique
    \(\delta\)-ring map \(\mathcal{R}\to R\) compatible with \(\varphi\).
  \end{claim1}

  \begin{claim2}
    Let \(R\) be as in Claim 1.
    Let \(\sigma\) be the Frobenius lift of \(R\). then \(x \in R\) is invertible
    if and only if \(\sigma(x)\) is.
  \end{claim2}

  \begin{proof}[Proof of Claim 2]
    Since \(\sigma(x)=x^{p}+p\delta(x)\), and \(R\) is \(p\)-adicaly complete,
    \(\sigma(x)\) is invertible if and only if \(x^{p}\) is invertible. But
    \(x^{p}\) is invertible if and only if \(x\) is.
  \end{proof}

  Applying Claim 2 to \((R,\sigma)=(\mathcal{R},\phi)\), we have
  \(\mathcal{R}=\mathcal{R}[(\det A_1)^{-1}]\).

  \begin{proof}[Proof of Claim 1]
    The universal property of
    \(\mathcal{R}_1\) implies that \(\varphi\) canonically factors through
    \(\mathcal{R}_{1}\) as a homomorphism of \(\delta\)-rings. By
    Claim 2, as \(\varphi(\det A_1)\) is invertible in \(R\),
    \(\sigma^{m}(\varphi(\det A_1))\) are all invertible in \(R\). Thus \(\varphi\)
    canonically factors through a \(\delta\)-ring homomorphism
    \(\mathcal{R}_{1}[\phi^{m}(\det A_1)^{-1}:m\geq 0] \to R\). Passing to the
    completion finishes the argument.
  \end{proof}

  Since \(\mathcal{R}\) is a completion of a localization of a
  polynomial ring (with infinitely many variables), \(\mathcal{R}\) is an integral
  domain. Therefore the theorem holds for the Laurent polynomial
  \(\mathbf{f}(t)=\sum_{w\in\Delta\cap\mathbb{Z}^{d}} x_{w}t^{w}\) with
  coefficients in \(\mathcal{R}\).

  Now let \(f\) be as in the statement of the theorem.
  By construction, there is a homomorphism \(\mathcal{R}_0 \to R\) sending
  \(\mathbf{f}\) to \(f\). Since \(R\) is equipped with a \(p\)-Frobenius and
  satisfies the hypotheses of Claim 1,
  we get a ring homomorphism \(\Psi:\mathcal{R} \to R\) compatible with the
  Frobenii on \(\mathcal{R}\) and \(R\). As the formal
  group laws \(F_{f}\), \(F_{\mathbf{f}}\) are defined by coefficients of
  expansions, we have \(\Psi_{\ast}\Phi_{\mathbf{f}}=\Phi_{f}\).
  Since \(\mathcal{R}\) is the completion of a localization of a polynomial
  ring, we know \(\mathcal{R}\) and \(\mathcal{R}/p\) are domains. Let
  \(\mathbb{D}^{\ast}(\Gamma_{\mathbf{f}})\) be the covariant Dieudonné crystal of the
  reduction of \(\Phi_{\mathbf{f}}\). This is a special sheaf on the big
  crystalline site \(\mathrm{CRIS}((\mathcal{R}/p)/\mathbb{Z}_{p})\).

  We have a commutative diagram
  \begin{equation*}
    \begin{tikzcd}
      \mathcal{R} \ar[r,"\Psi"] \ar[d] & R \ar[d] \\
      \mathcal{R}/p \ar[r, "\psi"] & R/p
    \end{tikzcd}.
  \end{equation*}
  Since \(\Phi_{\mathbf{f}}\) is formally smooth over \(R\), its formation
  commutes with base change. We have
  \(\Gamma_{\mathbf{f}}\otimes_{\mathcal{R}/p}R/p=\Phi_{\mathcal{R}}\otimes_{\mathcal{R}}R/p=\Gamma_{f}\).
  Therefore
  \(\mathbb{D}^{\ast}(\Gamma_f)=\mathrm{Spec}(\psi)^{\ast}\mathbb{D}^{\ast}(\Gamma_{\mathbf{f}})\)
  by \cite[(1.3.3.4)]{berthelot-breen-messing:crystalline-dieudonne-theory-2}
  and the definition of the pull-back
  functor~\cite[p.~30]{berthelot-breen-messing:crystalline-dieudonne-theory-2}.
  Since the Frobenius action on \(\mathbb{D}^{\ast}(\Gamma_f)\) is induced
  from that of \(\mathbb{D}^{\ast}(\Gamma_{\mathbf{f}})\), and we have checked
  that in the ``universal case'' \(\beta_{1}=\alpha\), the theorem for
  \(R\) and \(f\) follows from the theorem for \(\mathcal{R}\) and
  \(\mathbf{f}\) by base change.
\end{proof}

We finish with a remark on the relation between the Dieudonné module and rigid
cohomology.

\begin{remark}[Relation with rigid cohomology]%
  \label{remark:relation-with-rigid-cohomology}
  So far we have been completely ignoring the geometric meaning of
  \(\mathcal{C}(\Gamma_{f})\). In this remark we explain how to relate \(\mathcal{C}(\Gamma)\) to
  quantities with geometric meaning. In addition to the hypotheses above we assume further
  that \(R\) is a noetherian ring. Then by
  Lemma~\ref{lemma:formal-lie-cohomology}, we can identify \(\Gamma_{f}\) with
  an Artin--Mazur type formal group functor.

  Let \(X\) be closure of \(f=0\) in \(\mathbb{P}:=\mathbb{P}_{\Sigma,R}\). Let
  \(U = \mathbb{P} - X\).
  Let \(\mathbb{P}_{0}\), \(X_0\), \(U_0\) be the reduction of \(\mathbb{P}\),
  \(X\), \(U\) modulo \(p\), respectively.
  Assume that \(X\) is flat (so the formation of its ideal sheaf commutes with
  base change). Then \(\Gamma_{f}\) is the Artin--Mazur formal group functor
  associated with the ideal sheaf of \(X\).
  The value of \(\mathbb{D}^{\ast}(\Gamma_f)\) at a perfect point \(x: R \to k\)
  is the Witt vector cohomology
  \begin{equation*}
    \mathrm{H}^{d}(\mathbb{P}_{0}\otimes_{R}k,\mathrm{Ker}\{ W\mathcal{O}_{\mathbb{P}_0\otimes_{R}k} \to W \mathcal{O}_{X_0\otimes_{R}k} \}).
  \end{equation*}
  The ``isogeny class'' of \(\mathcal{C}(\Gamma_{f}\otimes_{R}k)\) is then the slope \(<1\)
  part of the rigid cohomology group \(\mathrm{H}^{d}_{\text{rig,c}}(U_0\otimes_{R}k)\)
  with proper
  support~\cite[Theorem~1.2]{berthelot-bloch-esnault:witt-vector-cohomology}.

  Note that this is also  the
  slope \(<1\) part of the \((d-1)\)th rigid cohomology of the hypersurface
  \(X_{0}\otimes_{R,x}k\), since we have an exact sequence of vector spaces over
  \(W(k)[1/p]\):
  (assuming \(d \geq 2\) to avoid the trivial case)
  \begin{equation*}
    \mathrm{H}^{d-1}_{\text{rig}}(\mathbb{P}_0\otimes_{R}k) \to
    \mathrm{H}^{d-1}_{\text{rig}}(X_0 \otimes_{R}k) \to
    \mathrm{H}^{d}_{\text{rig},c}(U_0\otimes_{R}k) \to
    \mathrm{H}^{d}_{\text{rig}}(\mathbb{P}_0\otimes_{R}k)
  \end{equation*}
  and since
  \(\mathrm{H}^{d-1}_{\text{rig}}(\mathbb{P}_0\otimes_{R}k)\) and
  \(\mathrm{H}^{d}_{\text{rig}}(\mathbb{P}_0\otimes_{R}k)\) are isoclinic
  of slope \(d-1\), \(d\) respectively.

  Under the hypothesis that \(\alpha_1\) is invertible,  the
  limit \(\alpha\) then gives a way to construct a formula for the
  ``unit-roots'' of \(X_0\).
  Assuming \(k=\mathbb{F}_{q}\) is a finite field,
  Theorem~\ref{theorem:dieudonne-module-frob} then gives a way to extract the
  unit roots of the zeta functions of a flat family of (even singular)
  hypersurfaces in a possibly singular toric variety over \(k\).
\end{remark}

\bibliographystyle{plain}
\IfFileExists{/Users/dzhang/Nextcloud/projects/bibliographies.bib}%
{\bibliography{/Users/dzhang/Nextcloud/projects/bibliographies.bib}}%
{\bibliography{/home/dzhang/Nextcloud/projects/bibliographies.bib}}%
\end{document}